\documentclass[a4paper,11pt,reqno]{amsart}

\usepackage[tmargin=1.1in,bmargin=1.1in,rmargin=1.2in,lmargin=1.2in]{geometry}
\usepackage{enumitem}
\usepackage{amsmath}
\usepackage{amssymb}
\usepackage{amsthm}
\usepackage{amsfonts}
\usepackage{mathrsfs}
\usepackage{mathtools}
\usepackage{hyperref}
\usepackage[capitalize,nameinlink,noabbrev,nosort]{cleveref}

\usepackage{graphicx}
\newcommand*{\Scale}[2][4]{\scalebox{#1}{$#2$}}

\allowdisplaybreaks

\newcommand{\N}{\mathbb{N}}
\newcommand{\PP}{\mathbb{P}}
\newcommand{\R}{\mathbb{R}}

\newcommand{\tf}{\tilde{f}}
\newcommand{\tg}{\tilde{g}}
\newcommand{\tilh}{\tilde{h}}
\newcommand{\teta}{\tilde{\eta}}

\newcommand\ol[1]{\ensuremath{\overline{#1}}}
\newcommand\ul[1]{\ensuremath{\underline{#1}}}
\newcommand{\vp}{\varphi}

\theoremstyle{plain}
\newtheorem{thm}{Theorem}[section]
\newtheorem{lem}[thm]{Lemma}
\newtheorem{prop}[thm]{Proposition}
\newtheorem{cor}[thm]{Corollary}

\Crefname{thm}{Theorem}{Theorems}
\Crefname{lem}{Lemma}{Lemmas}

\theoremstyle{definition}
\newtheorem{defn}[thm]{Definition}
\newtheorem{rem}[thm]{Remark}
\newtheorem{eg}[thm]{Example}

\numberwithin{equation}{section}
\numberwithin{figure}{section}

\begin{document}
	
\author{Arvind Ayyer}
\address{Arvind Ayyer, Department of Mathematics,
  Indian Institute of Science, Bangalore  560012, India.}
\email{arvind@iisc.ac.in}

\author{Saham Sil}
\address{Saham Sil, Department of Mathematics,
	Indian Institute of Science, Bangalore  560012, India.}
\email{sahamsil@iisc.ac.in}

\date{\today}

\title{Factorised stationary states for a long range misanthrope process}

\begin{abstract}
The misanthrope process is an interacting particle system where particles move between neighbouring sites with hop rates depending only on the number of particles at the departure and arrival sites.
Motivated by a discretised version of the Hammersley--Aldous--Diaconis process,
we introduce a \textit{partially asymmetric long range misanthrope process} (PALRMP) 
on a finite one-dimensional lattice with periodic boundary conditions where particles can move between sites that are not necessarily neighbours, as long as there are no particles in between the departure and arrival sites.
In this model, each site $\ell$ has an inhomogeneous rate parameter $x_\ell$ associated to it, and the hop rate of a particle moving from site $k$ to site $\ell$ depends upon the parameter associated to the target site $x_\ell$, the direction the particle moves, and the number of particles at sites $k$ and $\ell$.
We also consider the \textit{homogeneous} PALRMP, where all the $x_\ell$'s are 1.
We find necessary and sufficient conditions on the hop rates under which the stationary distribution is of factorised form 
for both the PALRMP and the homogeneous PALRMP, as well as the extreme variants, namely the ones where the particle motion is totally asymmetric (TALRMP) and symmetric (SLRMP).
As an illustrative example, we study in detail the discrete Hammersley--Aldous--Diaconis process.
\end{abstract}


\maketitle

\section{Introduction}

Interacting particle systems \cite{liggett-ips-2005} are continuous time Markov processes on lattices or directed graphs, where particles hop between sites, following hop rates which depend only on local configurations.
They have been greatly studied, especially in the physics literature, as such models exhibit nontrivial phenomena, even on one dimensional lattices, and can be used to model systems with more complicated dynamics.

One such interacting particle system is the \emph{zero range process} (ZRP), introduced in \cite{spitzer-1970}, where particles hop between neighbouring sites with the hop rate being a function only of the number of particles at the departure site.
These systems always exhibit factorised stationary distributions, even on arbitrary directed graphs \cite{evans-hanney-2005}.
ZRPs have been shown to be particularly useful in the study of condensation~\cite{evans-waclaw-2014}.
The \emph{misanthrope process}, introduced in \cite{cocozza-1985}, is a generalisation of the ZRP, where the hop rates depend upon the number of particles at the departure and arrival sites.
If the hop rates satisfy certain equations, then the stationary probabilities  factorise or equivalently the stationary distribution is said to be of product form;
see \cref{sec:back} for the details.

Motivated by the Hammersley--Aldous--Diaconis (HAD) process~\cite{hammersley-1972,aldous-diaconis-1995} (see also the work by Ferrari--Martin~\cite{ferrari-martin-2009}) and its discrete variant we introduce in \cref{sec:HAD}, we define a natural generalisation of the misanthrope process.
We work with a finite lattice with periodic boundary conditions, and we allow particles to move between sites arbitrarily far apart so long as there are no particles in the sites between them.
In the most general situation, if a particle moves from site $k$ containing $m$ particles to site $\ell$ containing $n$ particles, the rate will be $x_\ell \, u(m, n)$ if the motion is clockwise and $q\, x_\ell \, u(m, n)$ if it is counterclockwise.
We call such a system a \emph{partially asymmetric long range misanthrope process} (PALRMP).
The main thrust of this work is to understand when such systems will exhibit factorised stationary distributions.

The plan of the article is as follows.
After discussing the background in \cref{sec:back}, we will define the PALRMP formally in \cref{sec:def palrmp} and state the main results of the paper there.
We will state some properties of its master equation in \cref{sec:mastereq}, which will be key to the proofs.
We will study the totally asymmetric variant ($q = 0$), which we call the TALRMP in \cref{sec:talrmp}, followed by the PALRMP in \cref{sec:palrmp} and the symmetric variant, which we call the SLRMP in \cref{sec:slrmp}.
We will study their homogeneous variants in the same order in \cref{sec:htalrmp,sec:hpalrmp,sec:hslrmp} respectively.
The models being classified are shown in \cref{fig:models}.
Finally, we will define the discretised version of the HAD process, which is an example of a TALRMP and study some of its properties in \cref{sec:HAD}.

\begin{figure}[h!]
\centering
\includegraphics[scale=1]{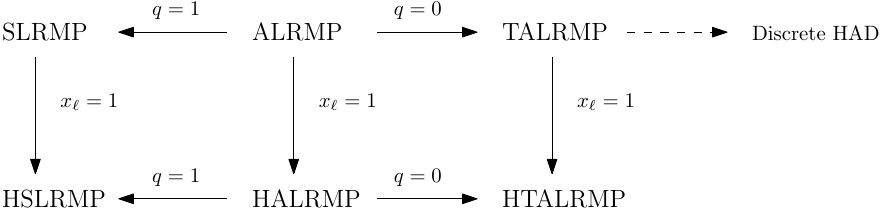}
\caption{All the models in this work. We use the prefix H to stand for homogeneous.}
\label{fig:models}
\end{figure}

\section{Background}
\label{sec:back}

We set the stage here for the kinds of processes we will be interested in.
We also set up the notation for the rest of the article.
Experts should feel free to skip this section.

In the following, $\N_0$ denotes the set of non-negative integers and for $n \in \N$ we denote $[n]$ to be set set $\{1,2,\dots,n\}$.
The rate functions $u$ considered are such that $u(m,n)$ is positive when $m\in \N$ and $n \in \N_0$ and $0$ otherwise, where $m$ and $n$ are the number of particles at the departure and arrival sites, respectively.

Throughout, we will work with interacting particle systems on 
a finitely one-dimensional lattice with periodic boundary conditions 
containing finitely many particles. 
There will only be one kind of particle, i.e. single species.
We will fix $L$ to be the number of sites and the number of particles will be conserved in the process, denoted by $N$.
We denote the state space as $\Omega_{L,N}$.
A configuration $\eta \in \Omega_{L,N}$ is specified by the number of particles at each site, $\eta = (\eta_1,\eta_2,\dots,\eta_L)$, where $\eta_\ell$ is the number of particles at site $\ell$ in the configuration.
The number of possible configurations is the number of solutions of $\eta_1+\eta_2+\cdots+\eta_L = N$, and thus
\begin{equation*}
    \left|\Omega_{L,N}\right| = \binom{L+N-1}{N} = \binom{L+N-1}{L-1}.
\end{equation*}
The positions of configurations will be indexed modulo $L$.

Consider a continuous time Markov process on a finite state space and let $\sigma$ be a probability distribution on the state space, which need not be stationary.
The \emph{probability current} due to a transition from $\eta \in \Omega_{L,N}$ to $\eta' \in \Omega_{L,N}$ is $\sigma(\eta) \times \text{rate}(\eta \to \eta')$.
The sum of the probability currents due to all transitions into $\eta$ is known as the \emph{incoming probability current}, and the sum of the probability currents due to all transitions out of $\eta$ is known as the \emph{outgoing probability current}.
For the stationary distribution of the process, which we will always denote by $\pi$, we refer to them as the \emph{stationary} incoming probability current and the \emph{stationary} outgoing probability current, respectively.

In a continuous time Markov process, a probability distribution $\pi$ is \textit{stationary} if it satisfies the \emph{master equation} or \emph{balance equation}, that is, for all $\eta \in \Omega_{L,N}$, the incoming probability current is equal to the outgoing probability current,
\begin{equation}
    \label{mastereq}
    \sum_{\eta' \in \Omega_{L,N}} \pi(\eta') \text{rate}(\eta' \to \eta) =
    \sum_{\eta' \in \Omega_{L,N}} \pi(\eta) \text{rate}(\eta \to \eta').
\end{equation}
If it turns out that the equality in \eqref{mastereq} holds for each summand, we say that \emph{detailed balance} holds.
If, on the other hand, for every $\eta'$, we can find an $\eta''$ such that
\begin{equation}
    \label{pairwise-balance}
    \pi(\eta') \text{rate}(\eta' \to \eta) = \pi(\eta) \text{rate}(\eta \to \eta''),
\end{equation}
and this map is invertible, we say that \emph{pairwise balance}~\cite{schutz-ramaswamy-barma-1996} holds.
If either pairwise balance or detailed balance holds, then the master equation \eqref{mastereq} is satisfied, as the corresponding transitions have the same probability currents.

We will work with irreducible Markov processes, so the stationary distribution exists and is unique.
We will denote the stationary distribution by $\pi \equiv \pi_{L,N}$.

\begin{defn}[Factorised probability distribution]
	\label{defn:factorised-steady-state}
	In an interacting particle system on $\Omega_{L,N}$, a probability distribution is said to be \emph{factorised} or \emph{of product form} if there exists a family of functions $g = (g_\ell)_{\ell=1}^L$,
    where $g_\ell : \N_0 \to \R_{> 0}$, such that the probability of a configuration $\eta = (\eta_1,\eta_2,\dots,\eta_L) \in \Omega_{L,N}$
    is given by
	\begin{subequations}
        \label{eqn:factorised-prob-dist}
        \begin{equation}
    		\sigma_{L,N}(\eta) = \dfrac{1}{Z_{L,N}^g} \prod_{\ell = 1}^L g_\ell(\eta_\ell),
    	\end{equation}
	where $Z_{L,N}^g$ is the normalization constant and is given by
        \begin{equation}
		      Z_{L,N}^g = \sum_{(\eta_1,\dots,\eta_L) \in \Omega_{L,N}} 
		      \left(\prod_{\ell=1}^L g_\ell(\eta_\ell)\right).
        \end{equation}
	\end{subequations}
    When there exists a family of functions $f = (f_\ell)_{\ell=1}^L$ such that the stationary distribution on $\Omega_{L,N}$ is given by \eqref{eqn:factorised-prob-dist} with respect to $f$, we say that the stationary distribution is factorised.
	The $f_\ell$'s are said to be \emph{one-point functions}.
\end{defn}

In case there is no ambiguity, we shall write $Z_{L,N}$ in place of $Z_{L,N}^g$.
We note that factorised stationary distributions do not give rise to product measures in our setting since the number of particles is conserved. 

The following result follows easily from \cref{defn:factorised-steady-state}.

\begin{prop}
    \label{prop:weight-func-rescale}
    If an interacting particle system on $\Omega_{L,N}$ has factorised stationary distribution determined by one-point functions $(f_\ell)_{\ell=1}^L$, then it is also factorised with respect to the one-point functions $(g_\ell)_{\ell=1}^L$, where $g_\ell(m) = a b^m f_\ell(m)$ for some $a,b \in \R_{> 0}$.
\end{prop}

In this article, we will only consider one-dimensional lattices on the ring.
So, for a rate $u$, we say that the stationary distribution corresponding to $u$ is factorised if there exists $f = (f_\ell)_{\ell=1}^\infty$, such that for all $L,N \in \N$, the stationary distribution on $\Omega_{L,N}$ is given by \eqref{eqn:factorised-prob-dist}.

\section{A partially asymmetric long range misanthrope process}
\label{sec:def palrmp}

Fix $q \geq 0$ and a rate function $u$.
We define the \emph{partial asymmetric long range misanthrope process} (PALRMP) on a one-dimensional ring lattice with $L$ sites, and associate a positive \emph{(inhomogeneous) rate parameter} to each site, denoted $x_\ell$ for site $\ell$.
Sites contain an arbitrary number of indistinguishable particles.
We identify moving right (resp. left) with going clockwise (resp. anti clockwise) in the ring lattice.
For a configuration of particles $\eta = (\eta_1,\eta_2,\dots,\eta_L)$ on $L$ sites, whenever a site $k$ is not empty, a particle moves from $k$ to a site to its right 
(resp. left), say site $\ell$, with rate $x_\ell \, u(\eta_k,\eta_\ell)$ (resp. $q \, x_\ell \, u(\eta_k,\eta_\ell)$), provided that the sites between site $k$ and site $\ell$, as we move right (resp. left), are empty.
See \cref{fig:possible-transitions} for an illustration of possible transitions.
When $x_\ell = 1$ for all sites, we shall refer to it as the \textit{homogeneous} PALRMP.
When $q = 0$, we refer to the process as the \emph{totally asymmetric long range misanthrope process} (TALRMP).
When $q = 1$, we refer to the process as the \emph{symmetric long range misanthrope process} (SLRMP).
We can similarly define the \emph{homogeneous} TALRMP and the \emph{homogeneous} SLRMP.

\begin{figure}[h]
	\centering
	\includegraphics[scale=0.3]{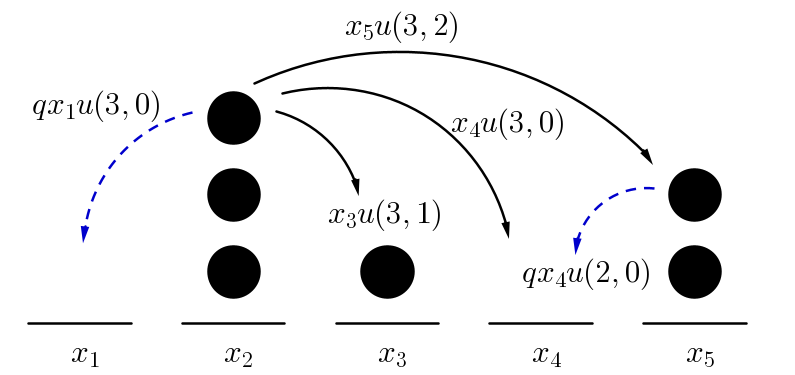}
	\caption{Some transitions, along with their rates, in the PALRMP from the configuration $(0,3,1,0,2) \in \Omega_{5,6}$.}
	\label{fig:possible-transitions}
\end{figure}

Consider a configuration $\eta \in \Omega_{L,N}$.
For $k,\ell \in [L]$, we denote $\eta^{k,\ell} \in \Omega_{L,N}$ and $\eta^\ell \in \Omega_{L,N+1}$ to be the configurations such that the number of particles at any site $a \in [L]$ is
\begin{subequations}
    \label{eqn:etakll}
    \begin{align}
        \label{eqn:eta-k,l}
        \eta^{k,\ell}_a &= \eta_a - \delta_{k,a} + \delta_{\ell,a},\\
        \label{eqn:eta-l}
        \eta^{\ell}_a &= \eta_a + \delta_{\ell,a}.
    \end{align}
\end{subequations}
If there is a transition due to a particle moving right (resp. left), from site $k$ to site $\ell$, then
\begin{equation}
    \text{rate}(\eta \to \eta^{k,\ell}) = x_\ell \, u(\eta_k,\eta_\ell) \ \ (\text{resp. } q \, x_\ell \, u(\eta_k,\eta_\ell)) .
\end{equation}
When there is a transition into $\eta$ where a particle moves right (resp. left), from site $k$ to site $\ell$, we denote
$\text{rate}(\eta^{\ell,k} \to \eta) = x_\ell \, u(\eta_k+1,\eta_\ell-1)$ (resp. $q \, x_\ell \, u(\eta_k+1,\eta_\ell-1)$).

\begin{prop}
	\label{prop:irred-LRMP}
	When all $u(m,n)$'s and $x_\ell$'s are positive, and with $q \geq 0$, the PALRMP is irreducible.
\end{prop}

\begin{proof}
	We consider only the transitions in the TALRMP, that is, 
	$q = 0$ and particles move only to the right.
	As these transitions are valid in any PALRMP, we will have proved the proposition.
	We are required to show that given any two configurations in $\Omega_{L,N}$, there is a sequence of transitions from one to the other.
	In order to do so, we consider the configuration $(N,0,\dots,0) \in \Omega_{L,N}$, where all particles are at the first site.
	
	For $\eta = (\eta_1,\eta_2,\dots,\eta_L)$, there are $\eta_L$ transitions where a particle at site $L$ moves right to site $1$.
    This gives the configuration $(\eta_1+\eta_L,\eta_2,\dots,\eta_{L-1},0)$.
	Then, as site $L$ does not contain particles in the resulting configuration, there are $\eta_{L-1}$ transitions where a particle moves from site $L-1$ to site $1$.
	Continuing in this manner, we get a sequence of transitions such that in the resulting configuration, all the particles are at site $1$.
	
	Now, from the configuration $(N,0,\dots,0)$, there are $\eta_L$ transitions where a particle moves right from site $1$ to site $L$, as all the sites in between have no particles.
    This gives the configuration $(N-\eta_L,0,\dots,0,\eta_L)$.
	Then, there are $\eta_{L-1}$ transitions where a particle moves from site $1$ to site $L-1$.
	Continuing in this manner, we get a sequence of transitions such that the resulting configuration is $\eta$.
	Hence, the TALRMP is irreducible, and so is the PALRMP.
\end{proof}

\begin{figure}[h]
	\centering
	\includegraphics[scale=0.3]{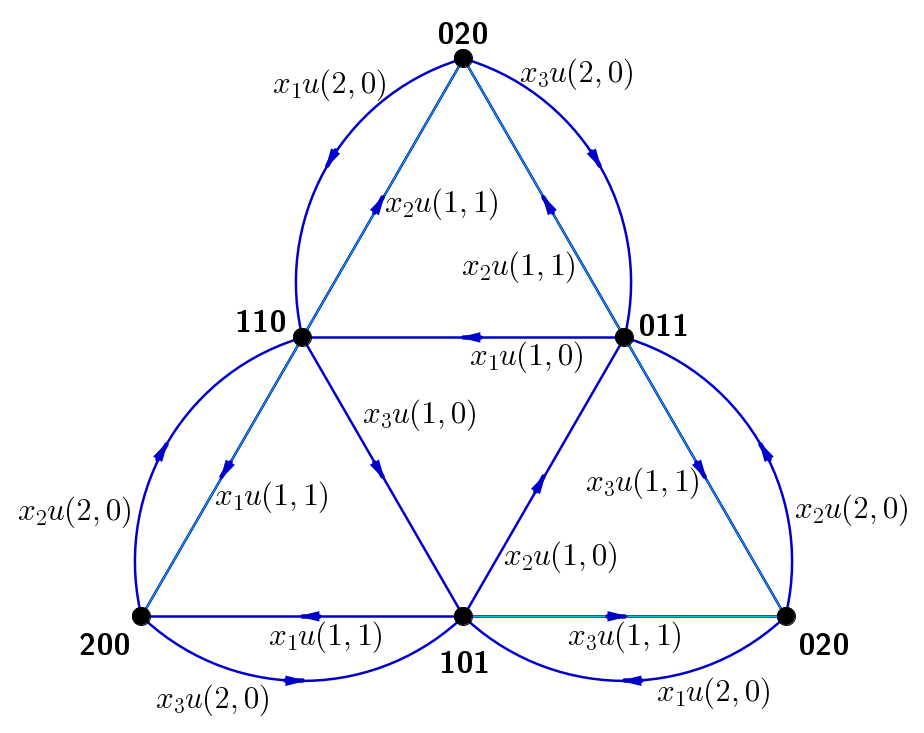}
	\caption{All transitions in the TALRMP on $\Omega_{3,2}$.}
    \label{fig:all-transitions-3,2}
\end{figure}

By \cref{prop:irred-LRMP}, the stationary distribution is unique, and hence we need only show that it satisfies the master equation \eqref{mastereq}.

We first state a symmetry property of the PALRMP.
Define the right shift map $S_R:\Omega_{L,N} \to \Omega_{L,N}$ by setting
\begin{equation*}
    S_R(\eta_1,\eta_2,\dots,\eta_L) = (\eta_L,\eta_1,\dots,\eta_{L-1}).
\end{equation*}

\begin{prop}[Translation covariance]
	\label{prop:trans-inv}
	Let $(X_t)$ be the Markov process for the usual PALRMP on $\Omega_{L,N}$, and $(Y_t)$, for the PALRMP on $\Omega_{L,N}$ where the $\ell$'th site has rate parameter $x_{\ell-1}$ modulo $L$.
	Suppose the probability distribution $\pi$ is stationary for $(X_t)$.
	Then, the probability distribution $\tilde{\pi} = \pi\circ S_R^{-1}$ is stationary for $(Y_t)$.
\end{prop}

\begin{proof}
	Consider a configuration $\eta \in \Omega_{L,N}$, along with sites $k$ and $\ell$, such that there is a transition where a particle moves from site $k$ to site $\ell$ in $(X_t)$, resulting in the configuration $\omega$.
    If the particle moves right (resp. left) in this transition, the rate of the transition is $x_\ell\, u(\eta_k,\eta_{\ell})$ (resp. $q \, x_\ell \, u(\eta_k,\eta_{\ell})$).
	Then, there is a transition from $S_R(\eta)$ to $S_R(\omega)$ where a particle moves from site $k+1$ to site $\ell+1$.
	The rate parameter corresponding to site $\ell+1$ in $(Y_t)$ is also $x_{\ell}$.
	Again, if the particle moved right (resp. left) in the initial system, then the particle moves right (resp. left) in the new system.
	So, the rate of the transition from $\eta$ to $\omega$ in $(X_t)$ is the same as the rate of the transition from $S_R(\eta)$ to $S_R(\omega)$ in $(Y_t)$.
	Hence, the stationary distribution in $(Y_t)$ is $\pi\circ S_R^{-1}$.
\end{proof}

The following result is a direct consequence of the translation covariance proved in \cref{prop:trans-inv}.

\begin{cor}
    \label{cor:one fn}
    Suppose that the stationary distribution for the PALRMP corresponding to $u$ is factorised with one-point functions $(f_\ell)_{\ell=1}^\infty$.
    Then, there exists a single function $f:\R_{>0}\times \N_0 \to \R$, such that for all $\ell \in \N$,
    $f_\ell(m) = f(x_\ell,m)$.
\end{cor}

So, from this point, when referring to the one-point function of a rate $u$ having a factorised stationary distribution for the PALRMP, we refer to $f$ only instead of $(f_\ell)_{\ell=1}^\infty$.

\begin{rem}
    \label{rem:one fn}
    Given a function $g:\R_{>0}\times \N_0 \to \R_{>0}$, we can define a probability distribution $\sigma$ on $\Omega_{L,N}$, for all $L,N \in \N$, by \eqref{eqn:factorised-prob-dist}, where $g_\ell(m) = g(x_\ell,m)$.
    Similarly, in the homogeneous case, given $g:\N_0 \to \R_{>0}$, we can define a probability distribution $\sigma$ on $\Omega_{L,N}$, for all $L,N \in \N$.
\end{rem}

We now state the main results of this work.

\begin{thm}
    \label{thm:palrmp-factorised}
	For $q\neq 1$, the stationary distribution corresponding to $u$ in the PALRMP is factorised if and only if there exists $\phi: \N_0 \to \R_{> 0}$, such that for all $n \in \N_0$,
	\begin{equation}
		\label{eqn:palrmp-factorised-cond}
		u(m,n) = \begin{cases}
		    \phi(n), \qquad &m \geq 1,\\
            0, \qquad &m=0.
		\end{cases}
	\end{equation}	
	The one-point function corresponding to $u$ is given by
	\begin{equation}
		\label{eqn:palrmp-factorised-weight}
		f(x,n) = x^n \prod_{i=1}^{n} \phi(i-1).
	\end{equation}
\end{thm}

\cref{thm:palrmp-factorised} for the TALMRP
is proved in \cref{sec:talrmp} and for the PALRMP in \cref{sec:palrmp}.

\begin{thm}
	\label{thm:hpalrmp-factorised}
	Suppose $q \neq 1$ in the homogeneous PALRMP.
    \begin{enumerate}[leftmargin = *]
        \item[(1)]
        The stationary distribution corresponding to $u$ in the homogeneous PALRMP is factorised if and only if for all $m,n \in \N$, $u$ satisfies the relations
        \begin{subequations}
        \label{eqn:hpalrmp-factorised-cond}
        	\begin{align}
        		\label{eqn:hpalrmp-factorised-cond-a}
            		u(m,n) =& \dfrac{u(1,m)}{u(1,n-1)}u(m+1,n-1)\cr
                    &+(u(m,1) - u(1,m)) - (u(n,1) - u(1,n)), \\ 
        		\label{eqn:hpalrmp-factorised-cond-b}            
        		      u(m,0) =& u(n,0). 
        	\end{align}
        \end{subequations}
    	The one-point function corresponding to $u$ is given by
    	\begin{equation}
    		\label{eqn:hpalrmp-factorised-weight}
    		f(n) = \prod_{i=1}^n u(1,i-1).
    	\end{equation}

        \item[(2)]
        The following are equivalent.
        \begin{enumerate}[leftmargin = *]
            \item[(a)]
            The stationary distribution corresponding to $u$ is factorised in the homogeneous PALRMP,

            \item[(b)]            
            There exist functions $b: \N_0\to\R_{> 0}$ and $c : \N_0 \to \R$, where $c(1) = 0$, such that for all $m,n \in \N$,
        	\begin{equation}
        		\label{eqn:hpalrmp-form-u-bc-end}
                u(m,n) = b(0)\prod_{k=1}^n\dfrac{b(m+k-1)}{b(n-k)}
                    + \sum_{\ell=0}^{n-1}\bigg((c(m+\ell) - c(n-\ell))\prod_{k=1}^\ell \dfrac{b(m+k-1)}{b(n-k)}\bigg).
        	\end{equation}
            In this case, for all $m \in \N$, $u(m,0) = b(0)$, $b(m) = u(1,m)$ and $c(m) = u(m,1) - u(1,m)$.
        	Then the one-point function corresponding to $u$ is given by
            \begin{equation}
            \label{eqn:hpalrmp-form-u-bc-end-form-f}
                f(n) = \prod_{i=1}^{n}b(i-1).
            \end{equation}

            \item[(c)]
            For all $m,n \in \N$
            \begin{subequations}                
                \label{eqn:hpalrmp-from-slrmp}
                \begin{align}
                    \label{eqn:hpalrmp-from-slrmp-a}
                    \dfrac{u(m,n)}{u(1,n)} &= \dfrac{u(n+1,m-1)}{u(1,m-1)},\\
                    \label{eqn:hpalrmp-from-slrmp-b}
                    u(m,n) - u(n,m) &= (u(m,1) - u(1,m)) - (u(n,1) - u(1,n)).
                \end{align}
            \end{subequations}
        \end{enumerate}
    \end{enumerate}
\end{thm}

\cref{thm:hpalrmp-factorised}(1) for the homogeneous TALRMP is proved in \cref{sec:htalrmp}.
\cref{thm:hpalrmp-factorised}(1) for the homogeneous PALRMP and \cref{thm:hpalrmp-factorised}(2) are proved in \cref{sec:hpalrmp}.

\begin{thm}
	\label{thm:(h)slrmp-factorised}
	The stationary distribution corresponding to u both in the SLRMP and in
    the homogeneous SLRMP is factorised if and only if for all $m,n \in \N$,
	\begin{equation}
		\label{eqn:(h)slrmp-factorised-cond}
		\dfrac{u(m,n)}{u(m,0)u(1,n)} = \dfrac{u(n+1,m-1)}{u(n+1,0)u(1,m-1)}.
	\end{equation}
	The one-point function corresponding to $u$ in the case of the SLRMP is given by
	\begin{equation}
		\label{eqn:slrmp-factorised-weight}
		f(x,n) = x^n \prod_{i=1}^n \dfrac{u(1,i-1)}{u(i,0)},
	\end{equation}
    and in the case of the homogeneous SLRMP, it is given by
    \begin{equation}
		\label{eqn:hslrmp-factorised-weight}
		f(n) = \prod_{i=1}^n \dfrac{u(1,i-1)}{u(i,0)}.
	\end{equation}
\end{thm}

The proofs of \cref{thm:(h)slrmp-factorised} for the SLRMP and the homogeneous SLRMP are given in \cref{sec:slrmp,sec:hslrmp}, respectively.

\begin{rem}
    \label{rem:(h)slrmp-similar-misanthrope}
    We note that \eqref{eqn:(h)slrmp-factorised-cond} is the same as \cite[Equation (2.3)]{cocozza-1985}.
    Cocozza-Thivent~\cite{cocozza-1985} showed that if the rate function $u$ satisfies \cite[Equation (2.3) and (2.4b)]{cocozza-1985}
    \begin{equation*}
    \begin{split}
        \dfrac{u(m,n)}{u(m,0)u(1,n)} =& \dfrac{u(n+1,m-1)}{u(n+1,0)u(1,m-1)}, \\
        u(n,m) - u(m,n) =& u(n,0) - u(m,0),
    \end{split}
    \end{equation*}
    for all $m \in \N$ and $n \in \N_0$, the stationary distribution for the misanthrope process on one-dimensional lattices is factorised.
\end{rem}

We note that if $u$ satisfies \eqref{eqn:palrmp-factorised-cond}, then it also satisfies \eqref{eqn:hpalrmp-factorised-cond} and \eqref{eqn:hpalrmp-from-slrmp}, which we expect, since if the stationary distribution is factorised in the case of the PALRMP, then it is also factorised in the case of the homogeneous PALRMP.
Similarly, if $u$ satisfies \eqref{eqn:hpalrmp-factorised-cond} (which is equivalent to \eqref{eqn:hpalrmp-from-slrmp} by \cref{thm:hpalrmp-factorised}(2), $u$ satisfies \eqref{eqn:(h)slrmp-factorised-cond}.
We note that \eqref{eqn:hpalrmp-from-slrmp-a} is similar to \eqref{eqn:(h)slrmp-factorised-cond}, when $u(m,0)$ is a constant function for $m \in \N$.

\begin{thm}
\label{thm:system-arb-weight}
	Fix a function $g: \N_0 \to \R_{> 0}$.
	\begin{enumerate}[leftmargin = *]
		\item[(1)]
		There exists a rate $u$ such that the stationary distribution corresponding to $u$ in the PALRMP is factorised and the one-point function corresponding to $u$ is given by
		\begin{equation}
			f(x,n) = x^n g(n).
		\end{equation}
		Furthermore, this rate $u$ is unique, upto a constant factor.
		
		\item[(2)]
		There exists a rate $u$ such that the stationary distribution corresponding to $u$ in the PALRMP is not factorised, but the stationary distribution corresponding to $u$ in the homogeneous PALRMP is factorised and the one-point function corresponding to $u$ is given by
		\begin{equation}
			f(n) = g(n).
		\end{equation}
		
		\item[(3)]
		There exists a rate $u$ such that the stationary distribution corresponding to $u$ in the homogeneous PALRMP is not factorised, but the stationary distribution corresponding to $u$ in the SLRMP and homogeneous SLRMP is factorised and the one-point function corresponding to $u$ is given by
		\begin{subequations}
            \begin{align}
    			f(x,n) &= x^n g(n),\\
    			f(n) &= g(n),
            \end{align}
		\end{subequations}
		for the SLRMP and the homogeneous SLRMP, respectively.
	\end{enumerate}
\end{thm}

\cref{thm:system-arb-weight}(1), (2) and (3) are proved in \cref{sec:palrmp,sec:hpalrmp,sec:hslrmp}, respectively.

\section{Master equation for the partially asymmetric long range misanthrope process}
\label{sec:mastereq}

In this section, we analyse the master equation for the PALRMP on $\Omega_{L,N}$ in terms of the PALRMP of smaller sized systems. 

For $\eta \in \Omega_{L,N}$, let $S_\eta \subseteq [L]$ be the set of sites that contain at least one particle.
Now, for a transition into $\eta$, the arrival site will contain at least one particle in $\eta$, and so is in $S_\eta$.
Similarly, for a transition out of $\eta$, the departure site will contain at least one particle in $\eta$, and so is in $S_\eta$.
Thus, any transition of $\eta$ involves at least one site in $S_\eta$.

Consider $S_\eta = \{s_1,s_2,\dots,s_K\}$, where $1\leq s_1 < s_2 < \cdots < s_K \leq L$. 
We index the sites $s_i$, where $i$ is taken modulo $K$.
Then, for sites $s_a$ and $\ell$, there are no particles in the sites between them if and only if $s_{a-1}\leq \ell \leq s_{a+1}$.
So, for a transition where a particle moves from site $s_a$, the arrival site is contained between the sites in $S_\eta$ either side of $s_a$.
Similarly, for a transition into $\eta$, where the particle moves to site $s_a$, the departure site $k$ of the particle would have to be contained within the sites between the sites either side of $s_a$ in $S_\eta$, that is, $s_{a-1} \leq k \leq s_{a+1}$.

In the case of the TALRMP, particles can move only to the right.
Then, for a site $s_a$, if a particle moves from site $s_a$ to site $\ell$, then $s_a < \ell \leq s_{a+1}$.
Conversely, if there is a transition out of $\eta$ where a particle moves to site $\ell$, then the departure site would be the site in $S_\eta$ closest to the left of site $\ell$.
Again, for transitions into $\eta$, where a particle moves to site $s_a$, the departure site $k$ is such that $s_{a-1} \leq k < s_a$.
Conversely, if there is a transition into $\eta$ where a particle moves from site $k$, the arrival site would be the site in $S_\eta$ closest to the right of site $k$.
Thus, for a site $\ell$ in the TALRMP, there is a unique transition out of $\eta$ where a particle moves to site $\ell$ and a unique transition into $\eta$ where a particle moves to site $\ell$.

Assume that $N \neq 0$ and $|S_\eta| = K$.
We define $\hat{\eta} \in \Omega_{K,N}$ to be the configuration where the $i$'th site has $\eta_{s_i}$ particles and rate parameter $x_{s_i}$.
Similarly, assuming that the $K < L$ and site $L$ is empty, we define $\ol{\eta} \in \Omega_{K+1,N}$, where, for $1\leq i \leq K$, the $i$'th site has $\eta_{s_i}$ particles and rate parameter $x_{s_i}$, and the $(K+1)$'th site has no particles and rate parameter $x_L$.
For example, in case of the configuration $\eta = (0,3,2,0,1,0)$, $\hat{\eta} = (3,2,1)$ and $\ol{\eta} = (3,2,1,0)$.
We note $\hat{\eta}$ is a configuration where all sites contain particles and in any transition into or out of $\hat{\eta}$, particles move between neighbouring sites.
If site $L$ of $\eta$ is empty, then $\ol{\eta}$ is a configuration with a single empty site.

\begin{prop}
	\label{prop:mu(eta)-transition-correspond}
	Consider $\eta \in \Omega_{L,N}$.
	\begin{enumerate}[leftmargin=*]
		\item[(1)]
		There is a bijective correspondence between the transitions of $\eta$ where both the departure and arrival sites are in $S_\eta$ and the transitions of $\hat{\eta}$. Moreover, the rates of the two corresponding transitions are equal.
		
		\item[(2)]
		Suppose that site $L$ of $\eta$ is empty. Then, there is a bijective correspondence between the transitions of $\ol{\eta}$ involving site $(K+1)$ and the transitions of $\eta$ involving site $L$. 
        Moreover, the rates of the two corresponding transitions are equal.
	\end{enumerate}
\end{prop}
\begin{proof}
	If $|S_\eta| =1$, then $\hat{\eta}$ would be a configuration with a single site, which has no transitions, and there are no transitions of $\eta$ where both the departure and arrival sites are in $S_\eta$.
	So, we only consider the case of $|S_\eta| > 1$.
	Let $|S_\eta| = K$ and $S_\eta = \{s_1,s_2,\dots,s_K\}$.	
	A transition out of (resp. into) $\hat{\eta}$, where a particle moves from site $a$ to site $b$ corresponds to the transition out of (resp. into) $\eta$ where a particle moves from site $s_a$ to site $s_b$.
	As the rate parameter of site $b$ in $\hat{\eta}$ and site $s_b$ in $\eta$ is $x_{s_b}$, the rates of both these transitions are equal.
	This proves part (1).
	
	Any transition out of (resp. into) $\ol{\eta}$ involving site $(K+1)$ is one where a particle moves to (resp. from) site $(K+1)$.
	Then, a transition out of (resp. into) $\ol{\eta}$, where a particle moves to (resp. from) site $(K+1)$ from (resp. to) site $a$, corresponds to a transition out of (resp. into) $\eta$ where a particle moves to (resp. from) site $L$ from (resp. to) site $s_a$, and by previous arguments in part (1), the rates of both these transitions are equal.
\end{proof}

\begin{prop}
	\label{prop:mu(eta)-transition-correspond-current}
	Consider $g: \R_{> 0}\times \N_0 \to \R$.
    Let $\sigma$ be the factorised probability distribution defined in terms of $g$ on $\Omega_{L, N}$ for all $L,N \in \N$, as in \cref{rem:one fn}.

    \begin{enumerate}[leftmargin = *]
		\item[(1)]		
		For $\eta \in \Omega_{L,N}$, the probability current due to a transition where both the departure and arrival sites are in $S_\eta$ is $(Z_{|S_\eta|,N}/Z_{L,N})\prod_{\ell \notin S_\eta} g(x_\ell,0)$ times the probability current of the corresponding transition of $\hat{\eta}$.
		
		\item[(2)]
		For $\eta \in \Omega_{L,N}$, suppose site $L$ is empty, so that $\ol{\eta}$ is well defined.
        Then, the probability current due to a transition involving site $L$ is $(Z_{|S_\eta|+1,N}/Z_{L,N})\prod_{\ell \notin (S_{\eta}\cup\{L\})} g(x_\ell,0)$ times the probability current of the corresponding transition of $\ol{\eta}$.
	\end{enumerate}
\end{prop}
\begin{proof}
    Since the proof of part (2) is similar to that of (1), we will only prove the latter.
    We note that part (1) does not apply if $|S_\eta| = 1$.
	So, we need only consider the case when $|S_\eta| > 1$.
	Suppose that $|S_\eta| = K$ and $S_\eta = \{s_1,s_2,\dots,s_K\}$.
    Then
    \begin{equation*}
        \sigma_{{K,N}}(\hat{\eta}) = \dfrac{1}{Z_{K,N}} \prod_{k=1}^K g(x_{s_k},\eta_{s_k}) = \dfrac{1}{Z_{K,N}} \prod_{\ell \in S_\eta}g(x_\ell,\eta_\ell).
    \end{equation*}
    As $\eta_{\ell} = 0$ for all $\ell \notin S_\eta$,
    \begin{equation*}
        \dfrac{\sigma_{{L,N}}(\eta)}{\sigma_{{K,N}}(\hat{\eta})} = \dfrac{Z_{K,N}}{Z_{L,N}} \prod_{\ell \notin S_\eta} g(x_\ell,0).
    \end{equation*}	
	Consider a transition out of $\eta$, where a particle moves right (resp. left) from site $s_a$ to $s_b$.
	In this case, the rate of the transition is $x_{s_b} \, u(\eta_{s_a},\eta_{s_b})$ (resp $q \, x_{s_b} \, u(\eta_{s_a},\eta_{s_b})$).
	By \cref{prop:mu(eta)-transition-correspond}(1), this corresponds to the transition out of $\hat{\eta}$, where a particle moves from site $a$ to site $b$, where the particle moves right (resp. left), with the same rate of transition.
	The probability current out of $\eta$ due to this transition is
	\begin{equation*}
		x_{s_b} \, u(\eta_{s_a},\eta_{s_b}) \, \sigma_{{L,N}}(\eta)
		= \left(\dfrac{Z_{K,N}}{Z_{L,N}} \prod_{\ell \notin S_\eta} g(x_\ell,0)\right) (x_{s_b} \, u(\eta_{s_a},\eta_{s_b})\sigma_{{K,N}}(\hat{\eta})),
	\end{equation*}
	where $x_{s_b} \, u(\eta_{s_a},\eta_{s_b}) \, \sigma_{{K,N}}(\hat{\eta})$ is the probability current out of $\hat{\eta}$ due to the corresponding transition.
	There is an additional factor of $q$ when the particle moves left.
    In either case, the probability current out of $\eta$ due to a transition where the particle moves to a site in $S_\eta$ is ${(Z_{K,N}/Z_{L,N})\prod_{\ell \notin S_\eta} g(x_\ell,0)}$ times the probability current of the corresponding transition out of $\hat{\eta}$.
	Similarly, the probability current into $\eta$, due to transitions with departure site in $S_\eta$, is $(Z_{K,N}/Z_{L,N})\prod_{\ell \notin S_\eta} g(x_\ell,0)$ times the probability current into $\hat{\eta}$ due to the corresponding transition.
	This proves part (1).
\end{proof}

The following example illustrates \cref{prop:mu(eta)-transition-correspond} and \cref{prop:mu(eta)-transition-correspond-current}.

\begin{eg}
	Suppose that the stationary distribution, $\pi$, corresponding to $u$ factorises in the PALRMP, with one-point function $f$.
	Let $\eta = (0,3,1,0,2,0) \in \Omega_{6,6}$.
	Then, $S_\eta = \{2,3,5\}$, and $\hat{\eta} = (3,1,2) \in \Omega_{3,6}$,  
	where sites $1$, $2$ and $3$ have  rate parameters $x_2$, $x_3$ and $x_5$, respectively.
	Similarly, $\ol{\eta} = (3,1,2,0) \in \Omega_{4,6}$, where site $4$ has rate parameter $x_6$ 
    in addition to the other three.
	
	The stationary probabilities of $\eta$, $\hat{\eta}$ and $\ol{\eta}$ are
	\begin{align*}
		\pi_{{6,6}}(\eta)
        &= \dfrac{1}{Z_{6,6}} f(x_1,0)f(x_2,3)f(x_3,1)f(x_4,0)f(x_5,2)f(x_6,0),\\
		\pi_{{3,6}}(\hat{\eta})
        &= \dfrac{1}{Z_{3,6}}f(x_2,3)f(x_3,1)f(x_5,2),\\
		\pi_{{4,6}}(\ol{\eta})
        &= \dfrac{1}{Z_{4,6}}f(x_2,3)f(x_3,1)f(x_5,2)f(x_6,0).
	\end{align*}
	So,
	\begin{equation*}
		\dfrac{\pi_{{6,6}}(\eta)}{\pi_{{3,6}}(\hat{\eta})} = \dfrac{Z_{3,6}}{Z_{6,6}} f(x_1,0)f(x_4,0)f(x_6,0) = \dfrac{Z_{3,6}}{Z_{6,6}} \prod_{\ell \notin S_\eta} f(x_\ell,0),
	\end{equation*}
	\begin{equation*}
		\dfrac{\pi_{{6,6}}(\eta)}{\pi_{{4,6}}(\ol{\eta})} = \dfrac{Z_{4,6}}{Z_{6,6}} f(x_1,0)f(x_4,0) = \dfrac{Z_{4,6}}{Z_{6,6}} \prod_{\ell \notin (S_\eta \cup \{6\})} f(x_\ell,0).
	\end{equation*}
    
    Recall $\eta^{k,\ell}$ from \eqref{eqn:eta-k,l}.
     Consider the transition from $\eta^{3,2} = (0,4,0,0,2,0)$ to $\eta$, with rate $x_3 \, u(4,0)$, where both the departure and arrival sites are in $S_\eta$.
    By \cref{prop:mu(eta)-transition-correspond}(1), this transition of $\eta$ corresponds to the transition of $\hat{\eta}$ from the configuration $\hat{\eta}^{2,1}=(4,0,2)$, where a particle moves left from site $1$ to site $2$, and the rate of this transition is also $x_3 \, u(4,0)$.
    So, the stationary probability current into $\eta$ due to this transition is
    \begin{align*}
        x_3 \, u(4,0) \pi_{6,6}(\eta^{3,2})
        &= x_3 \, u(4,0) \dfrac{1}{Z_{6,6}}f(x_1,0)f(x_2,4)f(x_3,0)f(x_4,0)f(x_5,2)f(x_6,0)\\
        &= \left(\dfrac{Z_{3,6}}{Z_{6,6}}f(x_1,0)f(x_4,0)f(x_6,0) \right) (x_3 \, u(4,0)\pi_{3,6}(\hat{\eta}^{2,1})),
    \end{align*}
    where $x_3 \, u(4,0)\pi_{3,6}(\hat{\eta}^{2,1})$ is the stationary probability current into $\hat{\eta}$ due to the corresponding transition in $\Omega_{3,6}$, as in \cref{prop:mu(eta)-transition-correspond-current}(1).
    
	Consider the transition from $\eta$ to $\eta^{5,6}$.
	By \cref{prop:mu(eta)-transition-correspond}(2), this transition corresponds to the transition from $\ol{\eta}$ to $\ol{\eta}^{3,4}$, and the rate of both of these transitions is $x_6 \, u(2,0)$.
	So, the stationary probability current out of $\eta$ due to this transition is
	\begin{align*}
		x_6 \, u(2,0) \pi_{{6,6}}(\eta)
		&=x_6 \, u(2,0)\dfrac{1}{Z_{6,6}} f(x_1,0)f(x_2,3)f(x_3,1)f(x_4,0)f(x_5,2)f(x_6,0)\\
		&= \left(\dfrac{Z_{4,6}}{Z_{6,6}}f(x_1,0)f(x_4,0)\right) (x_6 \, u(2,0)\pi_{4,6}(\ol{\eta})),
	\end{align*}
	where $(x_6 \, u(2,0)\pi_{4,6}(\ol{\eta}))$ is the stationary probability current out of $\ol{\eta}$ due to the corresponding transition in $\Omega_{4,6}$, as in \cref{prop:mu(eta)-transition-correspond-current}(2).
\end{eg}

We now introduce conditions for a factorised probability distribution to be the stationary distribution, and show that these conditions are both necessary and sufficient.

\begin{lem}
	\label{lem:me-new}
	Let $g: \R_{> 0}\times \N_0 \to \R$ be a function and $\sigma$ be the factorised probability distribution
	on $\Omega_{L,N}$, for all $L,N \in \N$, defined in terms $g$, as seen in \cref{rem:one fn}.
	Consider the PALRMP with rate $u$.
	If
    \begin{enumerate}[leftmargin = *]
        \item[(a)]
        for all $L,N \in \N$ with $N \geq L$ and for all $\eta \in \Omega_{L,N}$ such that $|S_\eta| = L$, $\eta$ satisfies the master equation for the PALRMP with probability distribution $\sigma$,
        
        \item[(b)]
        for all $L,N \in \N$ with $N \geq L-1$ and for all $\eta \in \Omega_{L,N}$ such that $|S_\eta| = L-1$, the incoming and outgoing probability currents due to transitions involving the empty site are equal,
    \end{enumerate}
    then $\sigma$ is the stationary distribution on $\Omega_{L,N}$, for all $L,N \in \N$, and $g$ is a one-point function corresponding to $u$.
\end{lem}

\begin{lem}
\label{lem:me-factorised-stationary}
	Suppose that the stationary distribution corresponding to $u$ is factorised in the PALRMP.
    \begin{enumerate}[leftmargin = *]
        \item[(1)]
        Consider a configuration $\eta \in \Omega_{L,N}$.
	    Then, both the stationary probability currents into and out of $\eta$ due to transitions where both the arrival and departure sites are in $S_\eta$ are equal.

        \item[(2)]
        Assume $N \geq L-1$ and consider a configuration $\eta \in \Omega_{L,N}$, where only one site is empty, say site $\ell$.
        Then, the stationary probability current into $\eta$ due to transitions where particles move from site $\ell$ is equal to the stationary probability current out of $\eta$ due to transitions where particles move to site $\ell$.

        \item[(3)]
        Consider a configuration $\eta \in \Omega_{L,N}$ and a site $\ell$ that is not in $S_\eta$.
        Then, the stationary probability current into $\eta$ due to transitions where particles move from site $\ell$ is equal to the stationary probability current out of $\eta$ due to transitions where particles move to site $\ell$.
    \end{enumerate}	
\end{lem}

\cref{lem:me-new,lem:me-factorised-stationary} thus imply that a factorised probability distribution $\sigma$ is stationary on $\Omega_{L,N}$, for all $L,N \in \N$ if and only if it satisfies conditions (a) and (b) of \cref{lem:me-new}.

\begin{proof}[Proof of \cref{lem:me-new}]
    Consider a configuration $\eta \in \Omega_{L,N}$, for some $L,N \in \N$.
    We show that $\eta$ satisfies the master equation.
    Suppose that $|S_\eta| = K$ and $S_\eta = \{s_1, s_2 , \dots , s_K\}$.    
    Any transition out of (resp. into) $\eta$ has departure (resp. arrival) site in $S_\eta$.
    So, for transitions out of (resp. into) $\eta$, either the arrival (resp. departure) site is in $S_\eta$, in which case both the arrival and departure sites of the transition are in $S_\eta$, or the arrival (resp. departure) site is not in $S_\eta$.
    
    Now, using \cref{prop:mu(eta)-transition-correspond-current}(1), summing over all the outgoing (rep. incoming) transitions from (resp. to) $\eta$ and $\hat{\eta}$, the probability current out of (resp. into) $\eta$ due to transitions with arrival (resp. departure) site in $S_\eta$, is $(Z_{K,N}/Z_{L,N})\prod_{\ell \notin S_\eta} f(x_\ell,0)$ times the probability current out of (resp. into) $\hat{\eta}$.
    As all sites of $\hat{\eta}$ contain particles, $\hat{\eta}$ satisfies the master equation \eqref{mastereq}, by condition (a).
    So, the probability current into $\eta$ due to transitions where both the departure and arrival sites are in $S_\eta$ is equal to the probability current out of $\eta$ due to transitions where both the departure and arrival sites are in $S_\eta$.
    
    Consider an empty site $\ell$ of $\eta$, which we consider to be site $L$ using \cref{prop:trans-inv}.
    In $\ol{\eta}$, only site $(K+1)$ is empty, and so,  by condition (b), the incoming and outgoing probability currents involving site $(K+1)$, in $\ol\eta$, are equal.
    By \cref{prop:mu(eta)-transition-correspond-current}(2), the probability current of a transition in $\eta$ involving site $L$ is $(Z_{K,N}/Z_{L,N})\prod_{\ell \notin (S_\eta \cup \{L\})} f(x_\ell,0)$ times the probability of the corresponding transition of $\ol{\eta}$.    
    So, the probability current into $\eta$ due to transitions where departure site is site $\ell$ is equal to the probability current out of $\eta$ due to transitions where arrival site is site $\ell$.  
    Then, summing over all empty sites, the probability current due to transitions out of $\eta$ where the arrival site is not in $S_\eta$ is equal to the probability current due to transitions into $\eta$ where the departure site is not in $S_\eta$. 

    Thus, the incoming probability current is equal to the outgoing probability current for $\eta$, that is $\eta$ satisfies the master equation.
    Hence, $\sigma$ is the stationary distribution of the PALRMP corresponding to $u$, and so, $g$ is a one-point function corresponding to $u$.
\end{proof}

\begin{proof}[Proof of \cref{lem:me-factorised-stationary}]
    Suppose that $\pi$ is the stationary probability distribution on $\Omega_{L,N}$, for all $L,N \in \N$, with one-point function $f$.
    \begin{enumerate}[leftmargin = *]
        \item[(1)]
        Considering the correspondence in the transitions of $\eta$ and $\hat{\eta}$, and using \cref{prop:mu(eta)-transition-correspond-current}(1) and arguments as in \cref{lem:me-new}, we can prove part (1).

        \item[(2)]
        Considering a cyclic relabelling of the sites, using \cref{prop:trans-inv}, suppose that site $L$ is empty, and all the other sites contain at least one particle in $\eta$.
        So, $S_\eta = [L-1]$ and transitions where both the arrival and departure sites are in $S_\eta$ are the transitions that do not involve site $L$.
        We note that any transition into or out of $\eta$ involves at least one site of $S_\eta$.
        So, the transitions into (resp. out of) $\eta$ involving site $L$ are the transitions where the particle moves from (resp. to) site $L$.
        By the master equation, the stationary probability current into $\eta$ is equal to the stationary probability current out of $\eta$.
        But, by part (1), the stationary probability current into $\eta$ due to transitions where a particle moves that was initially not at site $L$ is equal to the stationary probability current out of $\eta$ due to transitions where a particle moves to a site other than site $L$. 
        This proves part (2).

        \item[(3)]
        Considering the empty site to be site $L$, by \cref{prop:trans-inv}, and considering the stationary probability current due to corresponding transitions of $\eta$ and $\ol{\eta}$ from \cref{prop:mu(eta)-transition-correspond-current}(2), part (3) can be shown.
    \end{enumerate}
\end{proof}

For $N \geq L$, consider $\eta \in \Omega_{L,N}$ where all the sites contain at least one particle.
So, for any transition into or out of $\eta$, particles move between neighbouring sites.
Let $g: \R_{> 0}\times \N_0 \to \R$, and consider the factorised probability distribution $\sigma$, defined as in \cref{rem:one fn}, 
on the PALRMP, with rate $u$.
The outgoing probability current from $\eta$ corresponding to $\sigma$ is
\begin{equation}
	\label{eqn:palrmp-out-cur}
	\sigma(\eta) \sum_{\ell = 1}^L (x_\ell \, u(\eta_{\ell-1},\eta_\ell)+q \, x_\ell \, u(\eta_{\ell+1},\eta_\ell)).
\end{equation}
For any $1\leq \ell \leq L$, if a particle moves to site $\ell$ due to a transition, it was initially either at site $\ell-1$ or at site $\ell+1$.
In these cases, the transitions would have been from $\eta^{\ell,\ell-1}$ and $\eta^{\ell,\ell+1}$, respectively.
So, the incoming probability current to $\eta$ corresponding to $\sigma$ is
\begin{equation*}
    \begin{split}
        \sum_{\ell=1}^L&\bigg( \sigma( \eta^{\ell,\ell-1}) x_\ell \, u(\eta_{\ell-1}+1,\eta_\ell-1)+\sigma(\eta^{\ell,\ell+1})q\, x_\ell \, u(\eta_{\ell+1}+1,\eta_\ell-1)\bigg).
    \end{split}
\end{equation*}
We note that
\begin{equation*}
    \begin{split}
        \sigma(\eta^{\ell,\ell-1}) &= \dfrac{g(x_{\ell-1},\eta_{\ell-1}+1)}{g(x_{\ell-1},\eta_{\ell-1})}\dfrac{g(x_\ell,\eta_\ell-1)}{g(x_\ell,\eta_\ell)} \sigma(\eta),\\
        \sigma(\eta^{\ell,\ell+1}) &= \dfrac{g(x_\ell,\eta_\ell-1)}{g(x_\ell,\eta_\ell)} \dfrac{g(x_{\ell+1},\eta_{\ell+1}+1)}{g(x_{\ell+1},\eta_{\ell+1})}\sigma(\eta).
    \end{split}
\end{equation*}
Thus, the incoming probability current to $\eta$ corresponding to $\sigma$ is
\begin{equation}
	\label{eqn:palrmp-in-cur}
	\begin{split}
		\sigma(\eta)\sum_{\ell=1}^L \bigg(&\dfrac{g(x_{\ell-1},\eta_{\ell-1}+1)}{g(x_{\ell-1},\eta_{\ell-1})}\dfrac{g(x_\ell,\eta_\ell-1)}{g(x_\ell,\eta_\ell)} x_\ell \, u(\eta_{\ell-1}+1,\eta_\ell-1) \\
		&+\dfrac{g(x_\ell,\eta_\ell-1)}{g(x_\ell,\eta_\ell)} \dfrac{g(x_{\ell+1},\eta_{\ell+1}+1)}{g(x_{\ell+1},\eta_{\ell+1})} q\, x_\ell \, u(\eta_{\ell+1}+1,\eta_\ell-1) \bigg).
	\end{split}
\end{equation}

In the case when $\sigma$ is the stationary distribution on the PALRMP, \eqref{eqn:palrmp-out-cur} and \eqref{eqn:palrmp-in-cur} give the stationary outgoing and incoming probability currents for $\eta$, respectively.

\section{Totally asymmetric long range misanthrope process}
\label{sec:talrmp}

In this system, particles move only to the right (i.e. clockwise).
We recall that for $\eta \in \Omega_{L,N}$ and site $\ell \in [L]$, there is a unique transition into $\eta$ where a particle moves from site $\ell$, and a unique transition into $\eta$ where a particle moves to site $\ell$.

\begin{lem}
	\label{lem:talrmp-form-f}
	Suppose that the stationary distribution
    corresponding to $u$ is factorised in the TALRMP.
	Then, $u(m,0)$ is a constant function for $m \in \N$, and there exists a one-point function, $f$, corresponding to $u$ satisfying
	\begin{equation}
		\label{eqn:talrmp-factorised-weight-intermediate}
		f(x,n) = x^n f(x,0) \prod_{i=1}^n u(1,i-1).
	\end{equation}
\end{lem}
\begin{proof}
	Let $\tf$ be a one-point function corresponding to $u$, and $\pi$ the stationary distribution.
    
    We obtain a condition on $\tf$ by utilising \cref{lem:me-factorised-stationary}(2), and then utilise \cref{prop:weight-func-rescale} to obtain a one-point function satisfying \eqref{eqn:talrmp-factorised-weight-intermediate}.
    The form of the one-point function obtained in this manner will also satisfy condition (b) of \cref{lem:me-new}.
	
	For $N \geq L-1$, consider a configuration $\eta \in \Omega_{L,N}$ where there is a single empty site.
	Using \cref{prop:trans-inv}, suppose that the empty site is site $2$, and $\eta = (m,0,n,\eta_4,\dots,\eta_L)$, where the first three sites have rate parameters $x,y$ and $z$, respectively.	
	Then, the transition into $\eta$ involving site $2$ is from the configuration $\eta^{3,2} = (m,1,n-1,\eta_4,\dots,\eta_L)$, and the transition out of $\eta$ involving site $2$ is to the configuration $\eta^{1,2} = (m-1,1,n,\eta_4,\dots, \eta_L)$.
    By \cref{lem:me-factorised-stationary}(2), the stationary probability currents due to these transitions are equal.
    
	The stationary outgoing probability current due to transitions involving site $2$ is
	\begin{equation*}
		\pi(\eta)y \, u(m,0),
	\end{equation*}
	and the stationary incoming probability current due to transitions involving site $2$ is
	\begin{equation*}
		\pi(\eta^{3,2})z \, u(1,n-1) = \pi(\eta) \dfrac{\tf(y,1)}{\tf(y,0)} \dfrac{\tf(z,n-1)}{\tf(z,n)}z \, u(1,n-1).
	\end{equation*}
    We define $\tilh : \R_{> 0}\times \N \to \R$, such that
    \begin{equation}
        \label{eqn:inhom-define-tg}
        \tilh(w,k) = \dfrac{\tf(w,k)}{w\tf(w,k-1)}.
    \end{equation}
	Then, by \cref{lem:me-factorised-stationary}(2), equating the stationary incoming and outgoing probability currents due to transitions involving site $2$, for all $n,m \in \N$ and $x,y,z \in \R_{> 0}$,
	\begin{equation*}
		\tilh(z,n)u(m,0) = \tilh(y,1)u(1,n-1).
	\end{equation*}
	In the above equation, $y$ (resp. $z$) occurs only on the right (resp. left) hand side.
	So, $\tilh(y,1)$ does not depend on $y$ and $\tilh(z,n)$ does not depend on $z$.
	Thus, we can consider $h: \N \to \R$, such that for all $w\in \R_{> 0}$ and $k \in \N$,
    \begin{equation}
        \label{eqn:inhom-define-g}
        h(k) = \tilh(w,k).
    \end{equation}
	So, for all $n,m \in \N$,
	\begin{equation}
		\label{eqn:talrmp-me-one-empty}
		u(m,0)= \dfrac{h(1)}{h(n)} u(1,n-1).
	\end{equation}
	As $m$ occurs only on the left hand side of the above equation, $u(m,0)$ is a constant function for $m \in \N$.
    This proves the first part.
	
	Setting $m=1$ to \eqref{eqn:talrmp-me-one-empty} and expanding $h(n)$ as in \eqref{eqn:inhom-define-tg}, we have the recursive relation
	\begin{equation}
		\label{eqn:talrmp-tf-recur}
		\tf(z,n) = z\tf(z,n-1) \dfrac{h(1)}{u(1,0)} u(1,n-1).
	\end{equation}
	Solving the recursive relation,
	\begin{equation}		
		\label{eqn:talrmp-tf-recur-final}
		\tf(z,n) = z^n \left( \dfrac{h(1)}{u(1,0)}\right)^n \tf(z,0) \prod_{i=1}^n u(1,i-1).
	\end{equation}
	Hence, by \cref{prop:weight-func-rescale}, $f(x,n) = (u(1,0)/h(1))^n\tf(x,n)$ is a one-point function corresponding to $u$ satisfying \eqref{eqn:talrmp-factorised-weight-intermediate}.
\end{proof}

We note that in the last paragraph of the proof, we cannot also divide by $\tf(x,0)$ to get \eqref{eqn:palrmp-factorised-weight}, as the rate parameters may be different for different sites.

\begin{lem}
    \label{lem:talrmp-factorised-converse}
    If there exists $\phi:\N_0\to\R_{> 0}$, such that the rate $u$ satisfies \eqref{eqn:palrmp-factorised-cond}, then the stationary distribution corresponding to $u$ is factorised, with one-point function $f$ given by \eqref{eqn:palrmp-factorised-weight}.
\end{lem}
\begin{proof}
    Suppose that $\sigma$ is the factorised probability distribution corresponding to $f$.
    To prove that $\sigma$ is the stationary distribution, we need to show that $\sigma$ satisfies conditions (a) and (b) of \cref{lem:me-new}.

    For $N \geq L$, consider $\eta \in \Omega_{L,N}$, where all sites contain at least one particle.
    The probability current out of $\eta$, due to the transition where a particle moves to site $\ell$ from site $\ell -1$, is
    \begin{equation*}
        \sigma(\eta)x_\ell \, u(\eta_{\ell-1},\eta_\ell) = \sigma(\eta)x_\ell \, \phi(\eta_\ell).
    \end{equation*}
    The probability current into $\eta$, due to the transition where a particle moves from site $\ell$ to site $\ell+1$, expanding $f$ according to \eqref{eqn:palrmp-factorised-weight}, is
    \begin{equation*}
        \sigma(\eta)\dfrac{f(x_{\ell},\eta_{\ell}+1)}{f(x_{\ell},\eta_{\ell})}\dfrac{f(x_{\ell+1},\eta_{\ell+1} - 1)}{f(x_{\ell+1},\eta_{\ell+1})}x_{\ell+1} \,  u(\eta_{\ell}+1,\eta_{\ell+1}-1) = \sigma(\eta)x_\ell \, \phi(\eta_\ell).
    \end{equation*}
    So, the probability currents into and out of $\eta$ due to the transitions involving site $\ell$ are equal.
    Thus, we have pairwise balance as in \eqref{pairwise-balance} in this case, and so, $\eta$ satisfies the master equation.
    So, configurations where all sites contain at least one particle satisfy the master equation and $\sigma$ satisfies condition (a) of \cref{lem:me-new}.

    For $N \geq L-1$, consider a configuration $\xi \in \Omega_{L,N}$ with a single site that does not contain any particles.
    We focus on the case when $L \geq 3$.
    By considering a cyclic relabelling of the sites, using \cref{prop:trans-inv}, say that site $2$ is empty, and sites $1$ and $3$ contain $m$ particles and $n$ particles respectively.
    Thus, we can denote $\xi = (m,0,n,\xi_4,\dots,\xi_L)$.
    Suppose that the rate parameters of the first three sites are $x$ , $y$ and $z$, respectively.
    In any transition out of $\xi$ involving site $2$, a particle moves to site $2$.
    So, the only possible transition from $\xi$ involving site $2$ is the transition to $\xi^{1,2}$.
    So, the outgoing probability current due to transitions involving site $2$ is
    \begin{equation*}
        \sigma(\xi)y \, u(m,0) = y \, \phi(0) \dfrac{1}{Z_{L,N}}f(x,m)f(y,0)f(z,n) \prod_{\ell =4}^L f(x_\ell,\xi_\ell).
    \end{equation*}
    Now, $f(y,1) = y \, f(y,0)\phi(0)$, by \eqref{eqn:palrmp-factorised-weight}.
    So, from \eqref{eqn:eta-l}, the outgoing probability current due to transitions involving site $2$ is
    \begin{equation*}
        \dfrac{1}{Z_{L,N}}f(x,m)f(y,1)f(z,n) \prod_{\ell =4}^L f(x_\ell,\xi_\ell)= \dfrac{Z_{L,N+1}}{Z_{L,N}} \sigma ((m,1,n,\xi_4,\dots,\xi_L)) = \dfrac{Z_{L,N+1}}{Z_{L,N}} \sigma(\xi^2).
    \end{equation*}    
    In any transition into $\xi$ involving site $2$, a particle moves from site $2$.
    Then, this transition is from the configuration $\xi^{3,2}$.
    Now, $f(z,n) = z\phi(n-1)f(z,n-1)$, by \eqref{eqn:palrmp-factorised-weight}.
    So, from \eqref{eqn:eta-l}, the incoming probability current due to transitions involving site $2$ is
    \begin{equation*}
        \sigma(\xi^{3,2})zu(1,n-1) = z\phi(n-1)\frac{1}{Z_{L,N}}f(x,m)f(y,1)f(z,n-1)\prod_{\ell =4}^L f(x_\ell,\xi_\ell) = \dfrac{Z_{L,N+1}}{Z_{L,N}} \sigma(\xi^2).
    \end{equation*}
    So, the probability currents into and out $\eta$ due to the transitions involving site $2$ are equal.
    When $L = 2$, we can show that the probability currents are equal in a similar manner.
    Thus, $\sigma$ satisfies condition (b) of \cref{lem:me-new}.
    Hence, by \cref{lem:me-new} the stationary distribution corresponding to $u$ is factorised, with one-point function $f$ given by \eqref{eqn:palrmp-factorised-weight}.
\end{proof}

Note that in the proof of \cref{lem:talrmp-factorised-converse}, pairwise balance \eqref{pairwise-balance} holds for configurations where all sites contain particles and transitions involving empty sites.
It can shown that for any $\eta \in \Omega_{L,N}$ and site $\ell$, the stationary incoming current due to the unique transition where a particle moves from site $\ell$ is equal to the stationary outgoing current due to the unique transition where a particle moves to site $\ell$.

\begin{proof}[Proof of \cref{thm:palrmp-factorised} for TALRMP]
	We first consider the case when the stationary distribution corresponding to $u$, is factorised in the TALRMP.
	Let $\pi$ be the stationary distribution and $f$ be a one-point function corresponding to $u$ satisfying \eqref{eqn:talrmp-factorised-weight-intermediate}.
    We obtain conditions on $u$, by utilising the form of $f$ in \eqref{eqn:talrmp-factorised-weight-intermediate} and considering the master equation for configurations in which all the sites contain particles.
	
	Consider $N\geq L$ and a configuration $\eta \in \Omega_{L,N}$, where all the sites contain at least one particle.
	Then, the stationary outgoing current from $\eta$ is given by \eqref{eqn:palrmp-out-cur},
	\begin{equation*}
		\pi(\eta) \sum_{\ell=1}^L x_\ell \, u(\eta_{\ell-1},\eta_\ell).
	\end{equation*} 
    By \eqref{eqn:palrmp-in-cur}, as $q = 0$,  expanding $f$ according to \eqref{eqn:talrmp-factorised-weight-intermediate}, the stationary incoming probability current is
    \begin{equation*}
        \pi(\eta) \sum_{\ell=1}^L x_{\ell-1} \, u(1,\eta_{\ell-1}) \dfrac{1}{x_\ell \, u(1,\eta_\ell-1)} x_\ell \, u(\eta_{\ell-1}+1,\eta_\ell-1).
    \end{equation*}
    Rearranging the terms, the stationary incoming probability current is
    \begin{equation*}
        \pi(\eta) \sum_{\ell=1}^L x_{\ell}  \dfrac{u(1,\eta_{\ell})}{u(1,\eta_{\ell+1}-1)} u(\eta_{\ell}+1,\eta_{\ell+1}-1).
    \end{equation*}
    
	The stationary incoming and outgoing probability currents are polynomials in the $x_\ell$'s, and are equal by the master equation.
	So, for all $1\leq \ell \leq L$, the coefficients of $x_{\ell}$ are equal, that is,
	\begin{equation*}
		u(\eta_{\ell-1},\eta_{\ell}) =  \dfrac{u(1,\eta_{\ell})}{u(1,\eta_{\ell+1}-1)} u(\eta_{\ell}+1,\eta_{\ell+1}-1).
	\end{equation*}
	Thus, for all $m,n,k \in \N$,
	\begin{equation}
		\label{eqn:talrmp-me-no-empty}
		u(m,n) = \dfrac{u(1,n)}{u(1,k-1)}u(n+1,k-1).
	\end{equation}
	As the right hand side of \eqref{eqn:talrmp-me-no-empty} is not dependent on $m$, neither is the left hand side.
    So, $u$ is not dependent on the first argument.	
	Thus, there exists a function $\phi:\N_0\to \R_{> 0}$, such that $u(m,n) = \phi(n)$ for all $n \in \N_0$ and $m \in \N$, and $u(0,n) = 0$ for all $n \in \N_0$.
	
	The converse is shown by \cref{lem:talrmp-factorised-converse}.
\end{proof}

\section{Partially asymmetric long range misanthrope process}
\label{sec:palrmp}

\begin{lem}
	\label{lem:alrmp-iff-talrmp}
	The stationary distribution corresponding to $u$ is factorised in the TALRMP if and only if the stationary distribution corresponding to $u$ is factorised in the PALRMP for $q \neq 1$.
\end{lem}

\begin{rem}
	\label{rem:alrmp-q-remark}
	By \cref{lem:alrmp-iff-talrmp}, if we can show that the stationary distribution corresponding to $u$ factorises for some PALRMP, when $q \neq 1$, then the stationary distribution corresponding to $u$ factorises for all PALRMP.
	In the proof of \cref{lem:alrmp-iff-talrmp}, we show that if the stationary distribution factorises in the TALRMP, it also factorises in the SLRMP, but the converse is not true.
\end{rem}

\begin{proof}
    For a probability distribution $\sigma$ on $\Omega_{L,N}$ and configuration $\xi \in \Omega_{L,N}$, we define $\text{IRC}_{\sigma}(\xi)$ (resp. $\text{ORC}_{\sigma}(\xi)$) to be the stationary probability current into (resp. out of) $\xi$ due to transitions where particle move towards the right.
    We note that $\text{IRC}_{\sigma}(\xi)$ (resp. $\text{ORC}_{\sigma}(\xi)$) is the incoming (resp. outgoing) probability current into (resp. out of) $\xi$ in the TALRMP with the probability distribution $\sigma$.
    We can similarly define $\text{ILC}_{\sigma}(\xi)$ and $\text{OLC}_{\sigma} (\xi)$, due to particles moving left.
    So, $\xi$ satisfies the master equation for probability distribution $\sigma$ if and only if
    \begin{equation*}
        (\text{IRC}_{\sigma}(\xi) - \text{ORC}_{\sigma}(\xi)) + (\text{ILC}_{\sigma}(\xi) - \text{OLC}_{\sigma}(\xi))= 0.
    \end{equation*}
    
	Suppose that the stationary distribution corresponding to $u$ factorises in the TALRMP, with stationary distribution $\pi$ and one-point function $f$ satisfying \eqref{eqn:talrmp-factorised-weight-intermediate}.
    We show that $\pi$ is the stationary distribution corresponding to $u$ in the PALRMP.

    Consider $\eta = (\eta_1,\dots,\eta_L) \in \Omega_{L,N}$.    
    By the master equation for $\eta$ in the TALRMP,
    \begin{equation}
    	\label{eqn:alrmp-pc-me-eta}
        \text{IRC}_\pi(\eta) - \text{ORC}_\pi(\eta) = 0.
    \end{equation}

    Consider the system where site $\ell$ has rate parameter $x_{L+1-\ell}$, and the configuration $\teta = (\eta_L,\eta_{L-1},\dots,\eta_2,\eta_1) \in \Omega_{L,N}$. 
    This corresponds to reversing the configuration $\eta$.
    Considering the TALRMP on this new system, the stationary distribution corresponding to $u$ factorises with one-point function $f$, and so $\pi$ is the stationary distribution for the new system.
    As we have also changed the rate parameters,
    \begin{equation*}
        \pi(\eta) = \dfrac{1}{Z_{L,N}}f(x_1,\eta_1)f(x_2,\eta_2)\cdots f(x_L,\eta_L) = \dfrac{1}{Z_{L,N}}f(x_L,\eta_L)\cdots f(x_1,\eta_1) = \pi(\teta).
    \end{equation*}
    Considering the master equation of $\teta$,
    \begin{equation}
    	\label{eqn:alrmp-pc-me-oleta}
    	\text{IRC}_\pi(\teta) - \text{ORC}_\pi(\teta) = 0.
    \end{equation}

    There is a bijective correspondence between the transitions out of $\eta$ where particles move left and the transitions out of $\teta$, where particles moves right, namely, the transition out of $\eta$ where a particle moves left from site $k$ to site $\ell$ corresponds to the transition out of $\teta$, where a particle moves right from site $(L+1 - k)$ to site $(L+1 - \ell)$.
    Now, for all valid transition where a particle moves left from site $k$ to site $\ell$ in $\eta$,
    \begin{equation*}
    	\text{rate}(\eta \to \eta^{k,\ell}) = q \, x_\ell \, u(\eta_k,\eta_\ell) = q \times \text{rate}(\teta \to \teta ^{L+1 - k,L+1 - \ell})
    \end{equation*}
    Then, considering all transitions out of $\eta$ where particles move left,
    \begin{equation}
    	\label{eqn:alrmp-olc-equal}
    	\text{OLC}_\pi(\eta) = q \times \text{ORC}_\pi (\teta).
    \end{equation}
    Similarly,
    \begin{equation}
    	\label{eqn:alrmp-ilc-equal}
    	\text{ILC}_\pi(\eta) = q \times \text{IRC}_\pi (\teta).
    \end{equation}
    Then, from \eqref{eqn:alrmp-pc-me-eta}, \eqref{eqn:alrmp-pc-me-oleta}, \eqref{eqn:alrmp-olc-equal} and \eqref{eqn:alrmp-ilc-equal},
    \begin{equation*}
    	(\text{IRC}_{\pi}(\eta) - \text{ORC}_{\pi}(\eta)) + (\text{ILC}_{\pi}(\eta) - \text{OLC}_{\pi}(\eta))= 0,
    \end{equation*}
    that is, $\eta$ satisfies the master equation.
    Thus, $\pi$ is the stationary distribution corresponding to $u$ in the PALRMP.
    
    We now show the converse.
    Suppose that the stationary distribution, $\pi$, corresponding to $u$ factorises in the PALRMP, for some $q \neq 0,1$.
    We show that $\pi$ is also the stationary distribution corresponding to $u$ in the TALRMP by showing that all configurations satisfy the master equation.
    
    Consider $\eta = (\eta_1,\dots,\eta_L) \in \Omega_{L,N}$, and $\teta$ as above.
    Considering the master equation for $\eta$,
    \begin{equation}
    	\label{eqn:alrmp-new-me-eta}
    	(\text{IRC}(\eta) - \text{ORC}(\eta)) + q(\text{IRC}(\teta) - \text{ORC}(\teta)) = 0.
    \end{equation}
    Similarly, the master equation corresponding to $\teta$ is
    \begin{equation}
    	\label{eqn:alrmp-new-me-oleta}
    	q(\text{IRC}(\eta) - \text{ORC}(\eta)) + (\text{IRC}(\teta) - \text{ORC}(\teta)) = 0.
    \end{equation}
    From \eqref{eqn:alrmp-new-me-eta} and \eqref{eqn:alrmp-new-me-oleta}, as $q \neq 1$, we have
    \begin{equation}
    	\label{eqn:alrmp-left-current}
    	\text{IRC}(\eta) - \text{ORC}(\eta) = 0.
    \end{equation}
    So, $\eta$ satisfies the master equation in the TALRMP for probability distribution $\pi$, completing the proof.
\end{proof}

\begin{proof}[Proof of \cref{thm:palrmp-factorised}]
    By \cref{lem:alrmp-iff-talrmp}, for $q \neq 0,1$, the stationary distribution corresponding to $u$ is factorised in the PALRMP if and only if the stationary distribution corresponding to $u$ is factorised in the TALRMP, that is, $u$ satisfies \eqref{eqn:palrmp-factorised-cond}, and the one-point function is given by \eqref{eqn:palrmp-factorised-weight}.

    The converse, where we have to show that the probability distribution due to $f$ given by \eqref{eqn:palrmp-factorised-weight}, according to \cref{rem:one fn}, is indeed a stationary distribution, is also proven by \cref{lem:alrmp-iff-talrmp} and \cref{thm:palrmp-factorised} for the TALRMP, as the one-point functions are the same.   This completes the proof.
\end{proof}

\begin{proof}[Proof of \cref{thm:system-arb-weight}(1)]
	We first consider the case when $g(0) = 1$.
	Then, for $n \in \N_0$, let
    \begin{equation*}
        \phi(n) = \dfrac{g(n+1)}{g(n)},
    \end{equation*}
    and let $u$ be the rate defined by \eqref{eqn:palrmp-factorised-cond}.
    Then, the stationary distribution corresponding to $u$ in the PALRMP is factorised, and, as $g(0) = 1$, the one-point function is given by
    \begin{equation*}
        f(x,n) = x^n \prod_{i=1}^{n}\phi(i-1) = x^n \prod_{i=1}^n \dfrac{g(i)}{g(i-1)} = x^ng(n).
    \end{equation*}
	So, in this case, the assertion holds.
	
	If $g(0) \neq 1$, consider $\tg(m) = g(m)/g(0)$.
	Then, there exists a PALRMP with one-point function $\tf(x,n) = x^n \tg(x)$.
	As $g$ is a constant multiple of $\tg$, $f(x,n) = x^n g(n)$ is a one-point function of the PALRMP as well, by \cref{prop:weight-func-rescale}.
\end{proof}

\section{Symmetric long range misanthrope process}
\label{sec:slrmp}

We note that \eqref{eqn:(h)slrmp-factorised-cond} for $u$ is the same as \cite[Equation (2.3)]{cocozza-1985} and \cite[Equation (130)]{evans-hanney-2005}, which gives one of the two conditions for there to be a factorised stationary distribution in the misanthrope process,
as mentioned in \cref{rem:(h)slrmp-similar-misanthrope}.

\begin{lem}
	\label{lem:slrmp-form-f}
	Suppose that the stationary distribution corresponding to $u$ is factorisable in the SLRMP.
	Then, there exists a one-point function, $f$, corresponding to $u$ satisfying
	\begin{equation}
		\label{eqn:slrmp-factorised-weight-intermediate}
		f(x,n) = x^n f(x,0) \prod_{i=1}^n \dfrac{u(1,i-1)}{u(i,0)}.
	\end{equation}
\end{lem}

\begin{proof}
	Let $\tf$ be a one-point function corresponding to $u$, and $\pi$ the stationary distribution.
    We proceed as in the proof of \cref{lem:talrmp-form-f}, considering the configuration $\eta = (m, 0, n, \eta_4, \dots, \eta_L)$.
    The stationary outgoing probability current due to transitions involving site $2$ is
    \begin{equation*}
        \pi(\eta)y \, (u(m,0)+u(n,0)),
    \end{equation*}
    and the stationary incoming probability current due to transitions involving site $2$ is
    \begin{equation*}
        \pi(\eta)\left(\dfrac{\tf(y,1)}{\tf(y,0)} \dfrac{\tf(z,n-1)}{\tf(z,n)}z \, u(1,n-1)+\dfrac{\tf(y,1)}{\tf(y,0)} \dfrac{\tf(x,m-1)}{\tf(x,m)}x \, u(1,m-1)\right).
    \end{equation*}
    Defining $\tilh$ and $h$ according to \eqref{eqn:inhom-define-tg} and \eqref{eqn:inhom-define-g}, respectively,
    and setting $m = n$, we obtain
    \begin{equation}
    \label{eqn:salrmp-me-one-empty}
        u(n,0) = \dfrac{h(1)}{h(n)}u(1,n-1).
    \end{equation}

    So, we have the recursive relation
	\begin{equation*}
		\tf(x,n) = x\tf(x,n-1) h(1)\dfrac{u(1,n-1)}{u(n,0)}.
	\end{equation*}
	Solving the recursive relation,
	\begin{equation*}
		\tf(x,n) = x^n h(1)^n \tf(x,0) \prod_{i=1}^n \dfrac{u(1,i-1)}{u(i,0)}.
	\end{equation*}
	
	Hence, by \cref{prop:weight-func-rescale}, $f(x,n) = \tf(x,n)/(h(1)^n)$ is a one-point function corresponding to $u$ satisfying the condition in \eqref{eqn:slrmp-factorised-weight-intermediate}.
\end{proof}

We note that \eqref{eqn:talrmp-me-one-empty} and \eqref{eqn:salrmp-me-one-empty} differ, as there is an $m$ on the left hand side of \eqref{eqn:talrmp-me-one-empty}, and so we do not obtain $u(m,0)$ to be a constant function.

\begin{lem}
    \label{lem:slrmp-factorised-converse}
    If $u$ satisfies \eqref{eqn:(h)slrmp-factorised-cond} for all $m,n \in \N$, then the stationary distribution corresponding to $u$ is factorised, with one-point function $f$ given by \eqref{eqn:slrmp-factorised-weight}.
\end{lem}
\begin{proof}
	The proof is similar to that of \cref{lem:talrmp-factorised-converse}, using \cref{lem:me-new} to show that the probability distribution corresponding to $f$, according to \cref{rem:one fn}, is the stationary distribution.
	We can further show that for any arbitrary configuration, detailed balance  holds.
	The stationary incoming probability current due to the transition where the particle moves from site $k$ to site $\ell$ is equal to the stationary outgoing probability current due to the transition where the particle moves from site $\ell$ to site $k$.
\end{proof}

\begin{proof}[Proof of \cref{thm:(h)slrmp-factorised} for SLRMP]
	We first consider the case when the stationary distribution corresponding to $u$ is factorised in the SLRMP.
	Let $f$ be a one-point function corresponding to $u$ satisfying \eqref{eqn:slrmp-factorised-weight-intermediate}, with probability distribution $\pi$.
	
	For $N \geq L$, consider a configuration $\eta \in \Omega_{L,N}$, where all the sites contain at least one particle.
    Then, the stationary outgoing current from $\eta$ is given by \eqref{eqn:palrmp-out-cur},
	\begin{equation*}
		\pi(\eta)\sum_{\ell=1}^L x_\ell(u(\eta_{\ell-1},\eta_\ell)+u(\eta_{\ell+1},\eta_\ell)).
	\end{equation*}
	The stationary incoming probability current to $\eta$ is given by \eqref{eqn:palrmp-in-cur}.
    Using the explicit form of the one-point function due to \eqref{eqn:slrmp-factorised-weight-intermediate}, and rearranging the terms, the stationary incoming probability current into $\eta$ is,
	\begin{align*}
		\pi(\eta)\sum_{\ell=1}^L x_\ell\bigg(& \dfrac{u(1,\eta_\ell)}{u(\eta_\ell+1,0)}\dfrac{u(\eta_{\ell+1},0)}{u(1,\eta_{\ell+1}-1)}u(\eta_\ell+1,\eta_{\ell+1}-1)\\
		&+\dfrac{u(\eta_{\ell-1},0)}{u(1,\eta_{\ell-1}-1)} \dfrac{u(1,\eta_\ell)}{u(\eta_\ell+1,0)} u(\eta_\ell+1,\eta_{\ell-1}-1) \bigg).
	\end{align*}
	The stationary incoming and outgoing currents are polynomials in $x_\ell$'s, and are equal by the master equation.
    So, the coefficients of the two polynomials are the same, that is, for all $1\leq \ell \leq L$,
	\begin{equation*}
		\begin{split}			
			u(\eta_{\ell-1},\eta_\ell)+u(\eta_{\ell+1},\eta_\ell) =&
			\dfrac{u(1,\eta_\ell)}{u(\eta_\ell+1,0)}\dfrac{u(\eta_{\ell+1},0)}{u(1,\eta_{\ell+1}-1)}u(\eta_\ell+1,\eta_{\ell+1}-1)\\
		&+\dfrac{u(\eta_{\ell-1},0)}{u(1,\eta_{\ell-1}-1)} \dfrac{u(1,\eta_\ell)}{u(\eta_\ell+1,0)} u(\eta_\ell+1,\eta_{\ell-1}-1).
		\end{split}
	\end{equation*}
	Thus, for all $m,n,k \in \N$,
	\begin{equation}
		\label{eqn:slrmp-me-no-empty}
		\begin{split}			
			u(m,n)+u(k,n) =& \dfrac{u(1,n)}{u(n+1,0)}\dfrac{u(k,0)}{u(1,k-1)}u(n+1,k-1)\\
			&+\dfrac{u(1,n)}{u(n+1,0)}\dfrac{u(m,0)}{u(1,m-1)}u(n+1,m-1).
		\end{split}
	\end{equation}
	Substituting $k = m$ in \eqref{eqn:slrmp-me-no-empty}, for all $m,n \in \N$, we obtain \eqref{eqn:(h)slrmp-factorised-cond}.

    \cref{lem:slrmp-factorised-converse} shows the converse.
\end{proof}

\section{Homogeneous totally asymmetric long range misanthrope process}
\label{sec:htalrmp}

In the homogeneous variants of the PALRMP, all the rate parameters are $1$.
So, in a configuration $\eta$, the rate of a transition due to a particle moving right (resp. left), from site $k$ to site $\ell$ is $u(\eta_k,\eta_\ell)$ (resp. $q\,u(\eta_k,\eta_\ell)$). In this case, similar to \cref{cor:one fn} all the $f_\ell$'s will be equal, that is, there exists a single function $f: \N_0 \to \R$, such that for all sites $\ell \in \N$, $f_\ell(m) = f(m)$.
So, henceforth when referring to the one-point function corresponding to rate $u$, we refer to $f$ instead of $(f_\ell)_{\ell=1}^\infty$.
In the homogeneous case, given a one-point function $g$ corresponding to a hop rate $u$ with factorised stationary distribution in the homogeneous PALRMP, $f(m) = g(m)/g(0)$ is another one-point function corresponding to $u$, by \cref{prop:weight-func-rescale}.
So, for any rate $u$ which has factorised stationary distribution, there exists a one-point function $f$ corresponding to $u$ such that $f(0) = 1$.

\begin{lem}
	\label{lem:htalrmp-form-f}
	Suppose that the stationary distribution corresponding to $u$ is factorisable in the homogeneous TALRMP.
	Then, $u(m,0)$ is a constant function for $m \in \N$, and there exists a one-point function $f$, corresponding to $u$, satisfying \eqref{eqn:hpalrmp-factorised-weight}.
\end{lem}
\begin{proof}
	The proof is similar to that of \cref{lem:talrmp-form-f}.
\end{proof}

We note that, unlike \cref{lem:talrmp-form-f}, there is no additional factor of $f(0)$ in the expression for the one-point function that $f$ satisfies, as we can utilise \cref{prop:weight-func-rescale} to divide throughout.

The following lemma, \cref{lem:circular-sum-2}, is required in the proof of \cref{thm:hpalrmp-factorised}, when equating the stationary incoming and outgoing probability currents for a configuration all of whose sites contain particles.

\begin{lem}
	\label{lem:circular-sum-2}
	Consider a set $S$, a group $(G,\cdot)$, $M \in \N_0$ and $F: S\times S \to G$.
	Suppose that for all $L \geq M$ and $s_1,s_2,\dots,s_{L} \in S$, we have
	\begin{equation*}
		\prod_{\ell=1}^L F(s_\ell,s_{\ell+1}) = e_G,
	\end{equation*}
	where $s_{L+1} \equiv s_1$ and $e_G$ is the identity of $G$.
    Then, for all $p,q,r \in S$,
	\begin{equation*}
		F(p,q) = F(p,r) F(q,r)^{-1}.
	\end{equation*}
\end{lem}
\begin{proof}
	Let $p,q,r \in S$.
	Consider the configuration of length $M+1$ such that $s_1 = s_2 = \cdots = s_M = p$ and $s_{M+1} = r$.
	Similarly, we have the configuration of length $M+2$, where the first $M$ coordinates are $p$, the $(M+1)$'th coordinate is $q$ and the last coordinate is $r$.
	Then,
	\begin{equation*}
		F(p,p)^{M-1}F(p,r)F(r,p) = e_G = F(p,p)^{M-1}F(p,q)F(q,r)F(r,p).
	\end{equation*}
	
	Hence, for all $p,q,r \in S$,
	\begin{equation*}
		F(p,q) = F(p,r)F(q,r)^{-1}.
	\end{equation*}
\end{proof}

\begin{cor}
	\label{cor:circular-sum-2-real}
	If $F: \N\times \N \to \R$, such that for any tuple $(x_1,x_2,\dots,x_L)$ of length $L \geq 2$,
    \begin{equation*}
        \sum_{\ell=1}^L F(x_\ell,x_{\ell+1}) = 0,
    \end{equation*}
    then there exists $h : \N \to \R$, such that for all $n,m \in \N$,
	\begin{equation*}
		F(n,m) = h(n) - h(m).
	\end{equation*}
\end{cor}
\begin{proof}
	Defining $h(n) = F(n,1)$ suffices from \cref{lem:circular-sum-2}.
\end{proof}
We note that \cref{cor:circular-sum-2-real} has been proved in \cite[Appendix A]{evans-waclaw-2014} using a different approach.

\begin{proof}[Proof of \cref{thm:hpalrmp-factorised}(1) for homogeneous TALRMP]
    We first consider the case when the stationary distribution corresponding to $u$ is factorised in the homogeneous TALRMP.
	Let $f$ be the one-point function satisfying \eqref{eqn:hpalrmp-factorised-weight} and $\pi$ the stationary distribution.
    We note that \eqref{eqn:hpalrmp-factorised-cond-b} follows from \cref{lem:htalrmp-form-f}.
    
	Consider $N\geq L$ and a configuration $\eta \in \Omega_{L,N}$, where all sites contain at least one particle.
	Then, by \eqref{eqn:palrmp-out-cur}, the stationary outgoing current from $\eta$ is
	\begin{equation*}
	    \pi(\eta)\sum_{\ell=1}^L u(\eta_{\ell-1},\eta_\ell)
	\end{equation*}
	By \eqref{eqn:palrmp-in-cur}, the stationary incoming probability current, after expanding $f$ by \eqref{eqn:hpalrmp-factorised-weight}, is
    \begin{equation*}
        \pi(\eta) \sum_{\ell = 1}^L \dfrac{u(1,\eta_{\ell-1})}{u(1,\eta_\ell-1)}u(\eta_{\ell-1}+1,\eta_\ell-1).
    \end{equation*}
    Equating the currents, due to the master equation, as $\pi(\eta) \neq 0$,
    \begin{equation*}
        \sum_{\ell=1}^L \left( u(\eta_{\ell-1},\eta_\ell) - \dfrac{u(1,\eta_{\ell-1})}{u(1,\eta_\ell-1)}u(\eta_{\ell-1}+1,\eta_\ell-1) \right) = 0.
    \end{equation*}

    Then, by \cref{lem:circular-sum-2} and \cref{cor:circular-sum-2-real}, if we define
    \begin{equation*}
        F(m,n) = u(m,n) - \dfrac{u(1,m)}{u(1,n-1)}u(m+1,n-1),
    \end{equation*}
    for all $m,n \geq 1$, then
    \begin{equation*}
        F(m,n) = F(m,1) - F(n,1).
    \end{equation*}
    Then, as $u(\cdot,0)$ is a constant function by \cref{lem:htalrmp-form-f}, for all $m,n\in \N$
    \begin{align*}
        F(m,n) &= F(m,1) - F(n,1)\\
        &= \left(u(m,1) - \dfrac{u(1,m)}{u(1,0)}u(m+1,0)\right) - \left(u(n,1) - \dfrac{u(1,n)}{u(1,0)}u(n+1,0)\right)\\
        &= (u(m,1) - u(1,m)) - (u(n,1) - u(1,n))
    \end{align*}
    Thus, we obtain \eqref{eqn:hpalrmp-factorised-cond-a}, that is, for $m,n \in \N$,
    \begin{equation*}
        u(m,n) = \dfrac{u(1,m)}{u(1,n-1)}u(m+1,n-1)+(u(m,1) - u(1,m)) - (u(n,1) - u(1,n)).
    \end{equation*}
	
	We now show the converse, using \cref{lem:me-new}.
    Suppose $\sigma$ is the factorised probability distribution corresponding to $f$ given by \eqref{eqn:hpalrmp-factorised-weight}, according to \cref{rem:one fn}.

    For $N\geq L$, let $\zeta \in \Omega_{L,N}$ be a configuration, where all the sites contain at least one particle.
    By similar calculation as above, the incoming probability current is
    \begin{equation*}
        \sigma(\zeta) \sum_{\ell = 1}^L \dfrac{u(1,\zeta_{\ell-1})}{u(1,\zeta_\ell-1)}u(\zeta_{\ell-1}+1,\zeta_\ell-1),
    \end{equation*}
    and the outgoing probability current is
    \begin{equation*}
	    \sigma(\zeta)\sum_{\ell=1}^L u(\zeta_{\ell-1},\zeta_\ell).
	\end{equation*}
    Now, as
    \begin{equation*}
        u(\zeta_{\ell-1},\zeta_\ell) - \dfrac{u(1,\zeta_{\ell-1})}{u(1,\zeta_\ell-1)}u(\zeta_{\ell-1}+1,\zeta_\ell-1) = (u(\zeta_{\ell-1},1) - u(1,\zeta_{\ell-1})) - (u(\zeta_\ell,1)-u(1,\zeta_\ell)),
    \end{equation*}
    for all $1\leq \ell \leq L$, the difference between the incoming and outgoing probability current from $\zeta$ is $0$.
    So, $\zeta$ satisfies the master equation for probability distribution $\sigma$.

    For $N \geq L-1$, consider a configuration $\xi \in \Omega_{L,N}$, which contains a single site that does not have any particles.
    We first consider the case when $L \geq 3$.
    By considering a cyclic relabelling of the sites, using \cref{prop:trans-inv}, say that site $2$ is empty, and sites $1$ and $3$ contain $m$ particles and $n$ particles, respectively.
    So, $\xi = (m,0,n,\xi_4,\dots,\xi_L)$.

    Now, in any transition out of $\xi$ involving site $2$, a particle moves to site $2$.
    So, the only possible transition from $\xi$ involving site $2$ is the transition to $\xi^{1,2}$.
    So, as $u(\cdot,0)$ is a constant function and $f(1) = u(1,0) = u(1,0)f(0)$, the outgoing probability current due to transitions involving site $2$ is
    \begin{equation*}
        u(m,0)\sigma(\xi) = u(1,0)\dfrac{1}{Z_{L,N}}f(m)f(0)f(n) \prod_{\ell =4}^L f(\xi_\ell) = \dfrac{Z_{L,N+1}}{Z_{L,N}} \sigma (\xi^2).
    \end{equation*}    
    In any transition into $\xi$ involving site $2$, a particle moves from site $2$.
    Then, this transition is from the configuration $\xi^{3,2}$.
    So, as $f(n) = u(1,n-1)f(n-1)$ by \eqref{eqn:hpalrmp-factorised-weight}, the incoming probability current due to transitions involving site $2$ is
    \begin{equation*}
        u(1,n-1)\sigma(\xi^{3,2}) = u(1,n-1)\frac{1}{Z_{L,N}}f(m)f(1)f(n-1)\prod_{\ell =4}^L f(\xi_\ell) = \dfrac{Z_{L,N+1}}{Z_{L,N}} \sigma (\xi^2).
    \end{equation*}
    When $L=2$, we can show that the probability currents are equal in a similar manner.
    This shows the converse, and hence completes the proof.
\end{proof}

\section{Homogeneous partially asymmetric long range misanthrope process}
\label{sec:hpalrmp}

\begin{lem}
    \label{lem:halrmp-iff-htalrmp}
    The stationary distribution corresponding to $u$ is factorised in the homogeneous TALRMP if and only if the stationary distribution corresponding to $u$ is factorised in the homogeneous PALRMP for some $q \neq 0,1$.
\end{lem}
\begin{proof}
    The proof is similar to that of \cref{lem:alrmp-iff-talrmp}.
\end{proof}

\begin{proof}[Proof of \cref{thm:hpalrmp-factorised}(1)]
    By \cref{lem:halrmp-iff-htalrmp}, for $q\neq 0,1$, the stationary distribution corresponding to $u$ is factorised in the homogeneous PALRMP if and only if the stationary distribution corresponding to $u$ is factorised in the homogeneous TALRMP, that is, $u$ satisfies \eqref{eqn:hpalrmp-factorised-cond}, with the one-point function \eqref{eqn:hpalrmp-factorised-weight}.
    
    The converse, where we have to show that the probability distribution due to $f$ given by \eqref{eqn:hpalrmp-factorised-weight}, according to \cref{rem:one fn}, is indeed a stationary distribution, is also proven by \cref{lem:halrmp-iff-htalrmp} and \cref{thm:hpalrmp-factorised}, as the one-point functions are the same.
    This completes the proof.
\end{proof}

We now proceed to the proof of \cref{thm:hpalrmp-factorised}(2).

\begin{prop}
    \label{prop:hpalrmp-consistent}
    Let $b : \N_0 \to \R\setminus\{0\}$ and $c: \N_0 \to \R$, such that $c(1) = 0$.
    Then, for all $n \geq 1$,
    \[
    	\sum_{\ell=0}^{n-1}\bigg((c(1+\ell) - c(n-\ell))\prod_{k=1}^\ell \dfrac{b(k)}{b(n-k)}\bigg) = 0.
    \]
\end{prop}
\begin{proof}
	When $n=1$, the sum has a single term corresponding to $\ell = 0$, and $c(1+\ell) - c(n-\ell) = 0$.
    So, we consider the case when $n\geq 2$.
    With some change of variables and using $c(1)=0$, we obtain
    \begin{equation*}
        \sum_{\ell=0}^{n-1}\bigg(c(1+\ell) \prod_{k=1}^\ell \dfrac{b(k)}{b(n-k)}\bigg) = \sum_{\ell=1}^{n}\bigg(c(\ell) \prod_{k=1}^{\ell-1} \dfrac{b(k)}{b(n-k)}\bigg) = \sum_{\ell=2}^{n}\bigg(c(\ell) \prod_{k=1}^{\ell-1} \dfrac{b(k)}{b(n-k)}\bigg),
    \end{equation*}
    and
    \begin{equation*}
        \sum_{\ell=0}^{n-1}\bigg(c(n-\ell) \prod_{k=1}^\ell \dfrac{b(k)}{b(n-k)}\bigg) = \sum_{\ell=1}^{n}\bigg(c(\ell) \prod_{k=1}^{n-\ell} \dfrac{b(k)}{b(n-k)}\bigg) = \sum_{\ell=2}^{n}\bigg(c(\ell) \prod_{k=1}^{n-\ell} \dfrac{b(k)}{b(n-k)}\bigg).
    \end{equation*}
    So,
    \begin{equation*}
        \sum_{\ell=0}^{n-1}\bigg((c(1+\ell) - c(n-\ell))\prod_{k=1}^\ell \dfrac{b(k)}{b(n-k)}\bigg) = \sum_{\ell=2}^{n}\bigg(c(\ell) \bigg(\prod_{k=1}^{\ell-1} \dfrac{b(k)}{b(n-k)} - \prod_{k=1}^{n-\ell} \dfrac{b(k)}{b(n-k)}\bigg)\bigg).
    \end{equation*}
    But, for all $\ell \geq 2$,
    \begin{multline*}
        \prod_{k=1}^{\ell-1} \dfrac{b(k)}{b(n-k)} - \prod_{k=1}^{n-\ell} \dfrac{b(k)}{b(n-k)}
        = \dfrac{\prod_{k=1}^{\ell-1} b(k)}{\prod_{k = n-\ell+1}^{n-1} b(k)} - \dfrac{\prod_{k=1}^{n-\ell}b(k)}{\prod_{k=\ell}^{n-1} b(k)}\\
        = \dfrac{\left(\prod_{k=1}^{\ell-1} b(k)\right)\left(\prod_{k=\ell}^{n-1} b(k)\right) - \left(\prod_{k=1}^{n-\ell}b(k)\right)\left(\prod_{k = n-\ell+1}^{n-1} b(k)\right)}{\left(\prod_{k = n-\ell+1}^{n-1} b(k)\right)\left(\prod_{k=\ell}^{n-1} b(k)\right)} = 0,
    \end{multline*}
    as both the terms in the numerator are the same, completing the proof.
\end{proof}

\begin{proof}[Proof of \cref{thm:hpalrmp-factorised}(2)]
    We first show that (a) and (b) are equivalent.
    Suppose that (a) holds.
	Consider a rate $u$ satisfying \eqref{eqn:hpalrmp-factorised-cond}, such that $u(m,0) = \alpha > 0$ for all $m \geq 1$. Define functions $a,b,c : \N_0 \to \R$ by
	\[
        a(m) = u(m,1), \quad b(m) = u(1,m), \quad 
        \text{and} \quad c(m) = a(m) - b(m).
	\]
	We note that $a(0) = u(0,1) = 0$, $b(0) = u(1,0) = \alpha$, $c(0) = a(0) - b(0) =  -\alpha$ and $c(1) = u(1,1) - u(1,1) = 0$.
    So $u(m,0) = b(0)$ for all $m \geq 1$.
    
	For $m,n \geq 1$, from \eqref{eqn:hpalrmp-factorised-cond-a},
	\begin{equation}
		\label{eqn:hpalrmp-form-u-bc}
		u(m,n) = \dfrac{b(m)}{b(n-1)}u(m+1,n-1)+c(m) - c(n).
	\end{equation}
    We note that for $n = 1$, in \eqref{eqn:hpalrmp-form-u-bc}, as $b(0) = u(m+1,0)$
    \begin{equation*}
        u(m,1) = \dfrac{b(m)}{b(0)}u(m+1,0)+c(m) - c(1) = b(m)+c(m) =a(m),
    \end{equation*}
    that is, an identity relation.    
	
	Now, for $m \geq 1$ and $n \geq 2$, applying \eqref{eqn:hpalrmp-form-u-bc} twice,
	\begin{align*}
		u(m,n)
		=& \dfrac{b(m)}{b(n-1)}u(m+1,n-1)+c(m) - c(n)\\
		=& \dfrac{b(m)}{b(n-1)}\dfrac{b(m+1)}{b(n-2)}u(m+2,n-2)\\
        &+\dfrac{b(m)}{b(n-1)}(c(m+1) - c(n-1)) + (c(m) - c(n)).
	\end{align*}
	Continuing in this manner, applying \eqref{eqn:hpalrmp-form-u-bc} $j$ times, for $m \geq 1$ and $n \geq j$, we obtain
	\begin{equation}
		\label{eqn:hpalrmp-form-of-u-bc-j}
        \begin{split}            
		      u(m,n) =& u(m+j,n-j) \prod_{k=1}^j\dfrac{b(m+k-1)}{b(n-k)} \\
              &+ \sum_{\ell=0}^{j-1} \bigg((c(m+\ell) - c(n-\ell))\prod_{k=1}^\ell\dfrac{b(m+k-1)}{b(n-k)}\bigg).
        \end{split}
	\end{equation}
	Considering $j = n$ in \eqref{eqn:hpalrmp-form-of-u-bc-j}, as $u(n+m,0) = u(1,0) = b(0)$, for $m,n \in \N$,
	\begin{equation*}
		u(m,n) = b(0)\prod_{k=1}^n\dfrac{b(m+k-1)}{b(n-k)} +\sum_{\ell=0}^{n-1}\bigg((c(m+\ell) - c(n-\ell))\prod_{k=1}^\ell \dfrac{b(m+k-1)}{b(n-k)}\bigg),
	\end{equation*}
    that is, \eqref{eqn:hpalrmp-form-u-bc-end}.

    We show that the system is consistent.
	For $m,n \in \N$, \eqref{eqn:hpalrmp-factorised-cond-a} gives us an expression for $u(m,n)$ based on $u(m+1,n-1)$.
	So, we will have an expression of $u(1,m+n-1) = b(m+n-1)$ based on these substitutions.
	For the system to be consistent, it is required that this expression that is obtained is the same as $b(m+n-1)$ itself, that is, for all $n \in \N_0$, the expression for $u(1,n)$ obtained by \eqref{eqn:hpalrmp-form-u-bc-end} is $b(n)$.
	Now, by \eqref{eqn:hpalrmp-form-u-bc-end}, for $m = 1$,
	\begin{equation}
        \label{eqn:htalrmp-to-show-consistent}
		u(1,n) = b(0)\prod_{k=1}^n\dfrac{b(k)}{b(n-k)} +\sum_{\ell=0}^{n-1}\bigg((c(1+\ell) - c(n-\ell))\prod_{k=1}^\ell \dfrac{b(k)}{b(n-k)}\bigg).
	\end{equation} 
	Now,
	\[
		b(0)\prod_{k=1}^n\dfrac{b(k)}{b(n-k)} = b(0) \dfrac{b(1)b(2)\cdots b(n)}{b(n-1)b(n-2)\cdots b(0)} = b(n),
	\]
    and the sum in the right hand side of \eqref{eqn:htalrmp-to-show-consistent} is $0$, by \cref{prop:hpalrmp-consistent}.
	Thus, the system is consistent.

    Suppose (b) holds.
    So, there exist such functions $b$ and $c$, and $u$ is defined by \eqref{eqn:hpalrmp-form-u-bc-end}.
    Then, $u$ satisfies \eqref{eqn:hpalrmp-factorised-cond}.
    Again, as $b(m) = u(1,m)$, \eqref{eqn:hpalrmp-form-u-bc-end-form-f} follows from \eqref{eqn:hpalrmp-factorised-weight}.
    This shows that (a) and (b) are equivalent.

    We now show that (a) and (c) are equivalent.
    Suppose condition (a) holds.
    Then, condition (b) also holds.
    We first show \eqref{eqn:hpalrmp-from-slrmp-b}. 
    For $n = m$, \eqref{eqn:hpalrmp-from-slrmp-b} is an identity.
    We can assume, without loss of generality, that $m < n$.
    For $j = n-m$ in \eqref{eqn:hpalrmp-form-of-u-bc-j},
    \begin{equation}
    \label{eqn:hpalrmp-umn-unm-intermediate}
    \begin{split}
        u(m,n) =& u(n,m) \prod_{k=1}^{n-m}\dfrac{b(m+k-1)}{b(n-k)} \\
        &+ \sum_{\ell=0}^{n-m-1} \bigg((c(m+\ell) - c(n-\ell))\prod_{k=1}^\ell\dfrac{b(m+k-1)}{b(n-k)}\bigg).
    \end{split}
    \end{equation}
    We note that the coefficient of $u(n,m)$ is $1$.
    Consider $1\leq \ell_1 < \ell_2 \leq n-m-1$, such that $\ell_1+\ell_2 = n-m$.
    Then,
    \begin{equation*}
        c(m+\ell_1) - c(n-\ell_1) = c(n-\ell_2) - c(m+\ell_2),
    \end{equation*}
    and
    \begin{equation*}
        \prod_{k=\ell_1+1}^{\ell_2}\dfrac{b(m+k-1)}{b(n-k)} = 1.
    \end{equation*}
    So,
    \begin{equation*}
        (c(m+\ell_1) - c(n-\ell_1)) \prod_{k=1}^{\ell_1}\dfrac{b(m+k-1)}{b(n-k)} = - (c(m+\ell_2) - c(n-\ell_2)) \prod_{k=1}^{\ell_2}\dfrac{b(m+k-1)}{b(n-k)}.
    \end{equation*}
    In the case when $n-m$ is even, then for $\ell = (n-m)/2$, the term is $0$.
    So,
    \begin{equation*}
        \sum_{\ell=1}^{n-m-1} \bigg((c(m+\ell) - c(n-\ell))\prod_{k=1}^\ell\dfrac{b(m+k-1)}{b(n-k)}\bigg) = 0.
    \end{equation*}
    Then only the $\ell = 0$ term in \eqref{eqn:hpalrmp-umn-unm-intermediate} contributes to the sum, giving
    \begin{equation*}
        u(m,n) = u(n,m) +c(m) - c(n),
    \end{equation*}
    showing that \eqref{eqn:hpalrmp-from-slrmp-b} holds.
    Now, for $m,n \in \N$, from \eqref{eqn:hpalrmp-factorised-cond-a}, we have an expression of $u(n,m)$.
    Substituting it back into \eqref{eqn:hpalrmp-factorised-cond-b}, we obtain \eqref{eqn:hpalrmp-from-slrmp-a}.

    We now consider the case (c) holds, that is, \eqref{eqn:hpalrmp-from-slrmp} is satisfied.
    For $m=1$ in \eqref{eqn:hpalrmp-from-slrmp-a}, we have $u(1,n)/u(1,n) = u(n+1,0)/u(1,0)$ for all $n \in \N$, that is \eqref{eqn:hpalrmp-factorised-cond-b}.
    Now, from \eqref{eqn:hpalrmp-from-slrmp-a}, for all $m,n \in \N$
    \begin{equation*}
        u(n,m) = \dfrac{u(1,m)}{u(1,n-1)}u(m+1,n-1),
    \end{equation*}
    which, upon substitution into \eqref{eqn:hpalrmp-from-slrmp-b} gives \eqref{eqn:hpalrmp-factorised-cond-a}.
\end{proof}

\begin{rem}
    We note that the conditions in \eqref{eqn:hpalrmp-from-slrmp} are similar to \cite[Equation 2.3 and 2.4(b)]{cocozza-1985}.
    However, \eqref{eqn:hpalrmp-from-slrmp-a} cannot be replaced by \cite[Equation 2.3]{cocozza-1985}, unless we have the additional condition in \eqref{eqn:hpalrmp-from-slrmp} that $u(m,0)$ is a constant function for $m \in \N$.
\end{rem}

\begin{lem}
	\label{lam:htalrmp-arb-b}
	Given any $b: \N_0 \to \R_{> 0}$, there exists $c: \N_0\to\R$, such that $u$ as defined in \eqref{eqn:hpalrmp-form-u-bc-end} is a hop rate, such that the stationary distribution corresponding to $u$ is factorised in the homogeneous PALRMP, with one-point function $f$ given by \eqref{eqn:hpalrmp-form-u-bc-end-form-f}.
\end{lem}

\begin{proof}
    If such a $c$ existed, then the corresponding rate would satisfy \eqref{eqn:hpalrmp-factorised-cond}, in which case, the stationary distribution would be factorised, and as $u(1,m) = b(m)$, as in the proof \cref{thm:hpalrmp-factorised}(2), the one-point function would be determined by \eqref{eqn:hpalrmp-form-u-bc-end-form-f}.
	So, we are required to show that there exists $c: \N_0\to \R$ such that $u(m,n) > 0$ for all $m \in \N_0$ and $n \in \N$.
    
	Define $c(0) = -b(0)$ and $c(1) = 0$.
	Now, for $n = 0$, by \eqref{eqn:hpalrmp-form-u-bc-end}, $u(m,0) = b(0) > 0$.
	So, we need only show that for all $m,n \in \N$, $u(m,n) > 0$.
	Again, $u(1,1) = b(1) > 0$.
    Now, by \eqref{eqn:hpalrmp-form-u-bc-end}, evaluating $u(1,2)$ and $u(2,1)$,
	\begin{align*}
		u(1,2) &= b(2),\\
		u(2,1) &= b(2)+c(2).
	\end{align*}
	As $b(2) > 0$, for $u(m,n)$ to be positive whenever $m,n \in \N$ and $m+n \leq 3$, we require $-b(2) < c(2)$.
	So, we can define $c(2) = 1-b(2)$.
    Note that the choice of $c(2)$ is not unique.
	
	From \eqref{eqn:hpalrmp-form-u-bc-end}, for evaluating $u(m,n)$, we require the values $c(1),c(2),\dots,c(m+n-1)$.
	For $p \geq 3$, suppose we have defined $c(1), c(2),\dots,c(p-1)$ such that for all $m,n \in \N$ and $m+n \leq p$, $u(m,n) > 0$.
	Now, for $m,n \in \N$ such that $m+n = p+1$, by \eqref{eqn:hpalrmp-form-u-bc-end},
	\begin{align*}
		u(m,n)
		=& b(0)\prod_{k=1}^n\dfrac{b(m+k-1)}{b(n-k)} +\sum_{\ell=0}^{n-2}\bigg((c(m+\ell) - c(n-\ell))\prod_{k=1}^\ell \dfrac{b(m+k-1)}{b(n-k)}\bigg)\\
        &+\bigg((c(m+n-1) - c(1))\prod_{k=1}^{n-1} \dfrac{b(m+k-1)}{b(n-k)}\bigg)
	\end{align*}
	Thus, for all $m,n \in \N$, such that $m+n = p+1$, as $c(1) = 0$,
	\begin{equation}
		\label{eqn:htalrmp-form-of-u-bc-H}
		\begin{split}
			u(m,n) =& \left(\prod_{k=1}^{n-1} \dfrac{b(m+k-1)}{b(n-k)}\right)c(p)+b(0) \prod_{k=1}^n \dfrac{b(m+k-1)}{b(n-k)}\\
			&+\sum_{\ell=0}^{n-2} \bigg((c(m+\ell) - c(n-\ell))\prod_{k=1}^\ell \dfrac{b(m+k-1)}{b(n-k)}\bigg).
		\end{split}
	\end{equation}
	For $m = 1$ and $n = p$ in \eqref{eqn:htalrmp-form-of-u-bc-H}, by \cref{prop:hpalrmp-consistent}, and calculations as in \cref{thm:hpalrmp-factorised}(2) following \eqref{eqn:htalrmp-to-show-consistent}, $u(1,p) = b(p)$, and so, $u(1,p) >0$.
    So, we consider the case when $n < p$.
    We note that for $m = p$ and $n = 1$ in \eqref{eqn:htalrmp-form-of-u-bc-H}
    \begin{equation*}
        u(p,1) = c(p)+b(0)\dfrac{b(p)}{b(0)} = c(p)+b(p)
    \end{equation*}
    So, in this case, we require $c(p) > -b(p)$.
	In general, $u(m,n) > 0$ for all $m,n \in \N$ and $m+n =p+1$ if and only if $c(p)$ satisfies $(p-1)$ inequalities, obtained by taking the values of $n$ from $1$ to $(p-1)$, which corresponds to taking values of $m$ from $p$ to $2$.
	If $2 \leq m \leq p$ and $0 \leq \ell \leq n-2$, then $m + \ell \leq  m+n-2 = p-1$.
    Again, as $n \leq p-1$, for $0 \leq \ell \leq n-2$, $n - \ell \leq p-1$.
	So, there is a single occurrence of $c(p)$, with a positive coefficient, in each of the $p - 1$ inequalities.
	So, each of these inequalities gives a lower bound for $c(p)$.
	Thus, there are infinitely many solutions of $c(p)$, given $c(0), c(1), \dots, c(p-1)$, such that $u(m,n) > 0$ for all $m, n \in \N$ and $m+n \leq p+1$.
	
	So, by induction, we can define $c: \N_0 \to \R$ such that $u(m,n) > 0$ for all $m \in \N$ and $n \in \N_0$.
	Following the proof of \cref{thm:hpalrmp-factorised}(2),
    the stationary distribution corresponding to $u$ in the homogeneous PALRMP is factorised and the corresponding one-point function is given by \eqref{eqn:hpalrmp-form-u-bc-end-form-f}.
	However, $c$ is not unique, as in each step of the induction process, we are choosing $c(p)$ to satisfy a system of inequalities, each of which has infinitely many solutions.
\end{proof}

\begin{proof}[Proof of \cref{thm:system-arb-weight}(2)]
    We note that when the stationary distribution corresponding to the rate function $u$ factorises, then $u(m,1)$, for $m \geq 1$, is a constant function.
    Again, from the proof of \cref{thm:hpalrmp-factorised}(2), $u(m,1) = b(m)+c(m)$.
    
	We first consider the case when $g(0) = 1$.
	Then, for $n \in \N_0$, let
    \begin{equation*}
        b(n) = \dfrac{g(n+1)}{g(n)}.
    \end{equation*}
	By \cref{lam:htalrmp-arb-b}, there exists $c:\N_0 \to \R$ such that, $c(0) = -b(0)$, $c(1) = 0$, and for the rate $u$ defined according to \eqref{eqn:hpalrmp-form-u-bc-end}, the stationary distribution corresponding to $u$ in the homogeneous PALRMP factorises.
    Now, we can choose $c(2)$, such that $b(1)+c(1) \neq b(2)+c(2),$ in which case $u(1,1) \neq u(2,1)$.
    So, the stationary distribution corresponding to $u$ in the PALRMP does not factorise.
    By \eqref{eqn:hpalrmp-form-u-bc-end-form-f}, the corresponding one-point function in the homogeneous PALRMP is
	\begin{equation*}
	    f(n) = \prod_{i=1}^{n} b(i-1) = \prod_{i=1}^{n}\dfrac{g(i)}{g(i-1)} = \dfrac{g(n)}{g(0)} = g(n).
	\end{equation*}
	So, in this case, the assertion holds.
	
	When $g(0) \neq 1$, we proceed as in the proof of \cref{thm:system-arb-weight}, utilising \cref{prop:weight-func-rescale}.
\end{proof}

\section{Homogeneous symmetric long range misanthrope process}
\label{sec:hslrmp}

\begin{lem}
	\label{lem:hslrmp-form-f}
	For $q=1$, suppose that the stationary distribution corresponding to $u$ is factorisable in the homogeneous SLRMP.
	Then, there exists a one-point function, $f$, corresponding to $u$ satisfying \eqref{eqn:hslrmp-factorised-weight}.
\end{lem}
\begin{proof}
	The proof is similar to that of \cref{lem:slrmp-form-f}.
\end{proof}

\begin{lem}
	\label{lem:hslrmp-simplify-u}
	Suppose that $u$ is a rate function, such that for all $n,m \in \N$,
	\begin{equation}
		\label{eqn:hslrmp-simplify-u-initial}
		\begin{split}
            u(m,n)+u(n,m) = &\dfrac{u(1,n)}{u(n+1,0)}\dfrac{u(m,0)}{u(1,m-1)}u(n+1,m-1)\\
            &+\dfrac{u(1,m)}{u(m+1,0)}\dfrac{u(n,0)}{u(1,n-1)}u(m+1,n-1).
        \end{split}
	\end{equation}
	Then, $u$ satisfies \eqref{eqn:(h)slrmp-factorised-cond}.
\end{lem}

\begin{proof}
	For $n,m \in \N$, let
	\begin{equation}
		\label{eqn:hslrmp-simplify-u-def-H}
		h(n,m) = u(m,n) - \dfrac{u(1,n)}{u(n+1,0)}\dfrac{u(m,0)}{u(1,m-1)}u(n+1,m-1).
	\end{equation}
	We are required to show that $h$ is the constant $0$ function.
	
	By \eqref{eqn:hslrmp-simplify-u-initial}, for all $n,m \geq 1$,
	\begin{equation}		
		\label{eqn:hslrmp-simplify-u-symm-rel}
		h(n,m) + h(m,n) = 0.
	\end{equation}
	For $m \geq 2$, \eqref{eqn:hslrmp-simplify-u-def-H} gives
	\begin{align*}
		\dfrac{u(n+1,0)}{u(1,n)}\dfrac{u(1,m-1)}{u(m,0)}h(n,m)
		&= \dfrac{u(n+1,0)}{u(1,n)}\dfrac{u(1,m-1)}{u(m,0)}u(m,n) - u(n+1,m-1)\\
		&= -h(m-1,n+1),
	\end{align*}
	and so, for all $m \geq 2$ and $n \in \N$,
	\begin{equation}
		\label{eqn:hslrmp-simplify-u-non-symm-rel}
		\dfrac{h(n,m)}{u(m,0)u(1,n)}+\dfrac{h(m-1,n+1)}{u(n+1,0)u(1,m-1)} = 0.
	\end{equation}
	We note that in both the relations \eqref{eqn:hslrmp-simplify-u-symm-rel} and \eqref{eqn:hslrmp-simplify-u-non-symm-rel}, the sum of their coordinates of $h$ is $n+m$.
	So, we can group these relations by the sum of the coordinates.
    Consider that the sum of the coordinates is $p$.
	Then, it suffices to show that $h(1,p-1), h(2,p-2),\dots,h(p-1,1)$ are $0$ for all $p \geq 2$.
	We will only consider the relations in \eqref{eqn:hslrmp-simplify-u-symm-rel} where $n \leq m$ and the relations in \eqref{eqn:hslrmp-simplify-u-non-symm-rel} where $n \leq m-1$, and refer to them by their $n$ values, which, by our choice, is the smaller of the first coordinates of the relations. For example, we shall refer to $h(4,2)+h(2,4) = 0$ as the second relation due to \eqref{eqn:hslrmp-simplify-u-symm-rel}, as $2$ is the smaller of the first co-ordinates.
	
	We first consider the case when $p$ is odd.
	There are $(p-1)/2$ relations due to \eqref{eqn:hslrmp-simplify-u-symm-rel} and \eqref{eqn:hslrmp-simplify-u-non-symm-rel}, where $n$ varies from $1$ to $(p-1)/2$ in both cases.
	Let
	\begin{equation*}
		\ul v = (h(1,p-1),h(2,p-2),\dots, h(p-2,2),h(p-1,1)) \in \R^{p-1}.
	\end{equation*}
	Now, the relations due to \eqref{eqn:hslrmp-simplify-u-symm-rel} and \eqref{eqn:hslrmp-simplify-u-non-symm-rel} imply that $\ul v$ is orthogonal to certain vectors.
	We represent these relations using a matrix, $M_p \in \R^{(p-1)\times (p-1)}$, for which $\ul v \in \ker M_p$, defined as follows.
	For $1\leq i \leq (p-1)/2$, let the $(2i-1)$'th row correspond to \eqref{eqn:hslrmp-simplify-u-symm-rel} with $n = i$, and the $(2i)$'th row correspond to \eqref{eqn:hslrmp-simplify-u-non-symm-rel} with $n = i$.
	For example, when $p = 7$, we have $3$ relations due to \eqref{eqn:hslrmp-simplify-u-symm-rel} and \eqref{eqn:hslrmp-simplify-u-non-symm-rel} each, and so,
	\begin{equation*}
		M_7 = 
		\begin{pmatrix}
			1 & 0 & 0 & 0 & 0 & 1\\
			\Scale[0.75]{\dfrac{1}{u(1,1)u(6,0)}} & 0 & 0 & 0 & \Scale[0.75]{\dfrac{1}{u(1,5)u(2,0)}} & 0\\
			0 & 1 & 0 & 0 & 1 & 0\\
			0 & \Scale[0.75]{\dfrac{1}{u(1,2)u(5,0)}} & 0 & \Scale[0.75]{\dfrac{1}{u(1,4)u(3,0)}} & 0 & 0\\
			0 & 0 & 1 & 1 & 0 & 0\\
			0 & 0 & \Scale[0.75]{\dfrac{2}{u(1,3)u(4,0)}} & 0 & 0 & 0
		\end{pmatrix}.
	\end{equation*}
	In general, when $p$ is odd,
	\begin{equation*}
		M_p = 
		\begin{pmatrix}
			1 & 0 & 0 & \cdots & 0 & 0 & \cdots & 0 & 0 & 1\\
			\Scale[0.75]{\dfrac{1}{u(1,1)u(p-1,0)}} & 0 & 0 & \cdots & 0 & 0 & \cdots & 0 & \Scale[0.75]{\dfrac{1}{u(1,p-2)u(2,0)}} & 0\\
			0 & 1 & 0 & \cdots & 0 & 0 & \cdots & 0 & 1 & 0\\
			\vdots & \vdots & \vdots &  & \vdots & \vdots &  & \vdots & \vdots & \vdots\\
			0 & 0 & 0 & \cdots & 1 & 1 & \cdots & 0 & 0 & 0\\
			0 & 0 & 0 & \cdots & \Scale[0.75]{\dfrac{2}{u(1,\frac{p-1}{2})u(\frac{p+1}{2},0)}} & 0 & \cdots & 0 & 0 & 0
		\end{pmatrix}.
	\end{equation*}
	The columns specified in the middle of the above matrix are the $(p-1)/2$'th and $(p+1)/2$'th columns of $M_p$.
	From the last row, that is the $(p-1)$'th row, of $M_p$, we note that $h((p-1)/2,(p+1)/2) = 0$.
	Now, applying this to the $(p-2)$'th row, $h((p+1)/2,(p-1)/2) = 0$.
	Continuing in this manner, for all $1\leq n \leq p-1$, $h(n,p-n) = 0$.
	
	We now consider the case when $p$ is even.
	In this case, there are $p/2$ relations due to \eqref{eqn:hslrmp-simplify-u-symm-rel}, where $n$ varies from $1$ to $p/2$, and $(p-2)/2$ relations due to \eqref{eqn:hslrmp-simplify-u-non-symm-rel}, where $n$ varies from $1$ to $(p-2)/2$.
	We define the corresponding matrix $M_p$, such that the $(2i-1)$'th row of $M_p$ corresponds to \eqref{eqn:hslrmp-simplify-u-symm-rel} with $n = i$, and the $(2i)$'th row corresponds to \eqref{eqn:hslrmp-simplify-u-non-symm-rel} with $n = i$.
	For example, when $p = 6$,
	\begin{equation*}
		M_6 =
		\begin{pmatrix}
			1 & 0 & 0 & 0 & 1\\
			\Scale[0.75]{\dfrac{1}{u(1,1)u(0,5)}} & 0 & 0 & \Scale[0.75]{\dfrac{1}{u(4,1)u(0,2)}} & 0\\
			0 & 1 & 0 & 1 & 0\\
			0 & \Scale[0.75]{\dfrac{1}{u(2,1)u(0,4)}} & \Scale[0.75]{\dfrac{1}{u(3,1)u(0,3)}} & 0 & 0\\
			0 & 0 & 2 & 0 & 0
		\end{pmatrix}.
	\end{equation*}
	In general, when $p$ is even,
	\begin{equation*}
		M_p = 
		\begin{pmatrix}
			1 & 0 & \cdots & 0 & 0 & \cdots & 0 & 1\\
			\Scale[0.75]{\dfrac{1}{u(1,1)u(0,p-1)}} & 0 & \cdots & 0 & 0 & \cdots & \Scale[0.75]{\dfrac{1}{u(p-2,1)u(0,2)}} & 0\\
			0 & 1 & \cdots & 0 & 0 & \cdots & 1 & 0\\
			\vdots & \vdots &  & \vdots & \vdots & & \vdots & \vdots\\
			0 & 0 & \cdots & \Scale[0.75]{\dfrac{1}{u(\frac{p-2}{2},1)u(0,\frac{p+2}{2})}} & \Scale[0.75]{\dfrac{1}{u(\frac{p}{2},1)u(0,\frac{p}{2})}} & \cdots & 0 & 0\\
			0 & 0 & \cdots & 0 & 2 & \cdots & 0 & 0			
		\end{pmatrix}.
	\end{equation*}
	The columns specified in the middle of the above matrix are the $(p-2)/2$'th and $p/2$'th columns.
	From the last row, $h(p/2,p/2) = 0$, and following the same argument as before, $h(n,p-n) = 0$ for all $1\leq n \leq p-1$.
	Thus, $h$ is the constant zero function, proving the claim.
\end{proof}

\begin{proof}[Proof of \cref{thm:(h)slrmp-factorised} for homogeneous SLRMP]
	We consider the case when the stationary corresponding to $u$ is factorised, with one-point function $f$ satisfying the condition in \cref{lem:hslrmp-form-f}, namely \eqref{eqn:hslrmp-factorised-weight}, and $\pi$ the stationary distribution.
	
	For $N \geq L$, consider a configuration $\eta \in \Omega_{L,N}$, all of whose sites contain at least one particle.
	Then, by \eqref{eqn:palrmp-out-cur} the stationary outgoing probability current is given by
	\begin{equation*}
		\pi(\eta)\sum_{\ell=1}^L( u(\eta_{\ell-1},\eta_\ell)+u(\eta_{\ell+1},\eta_\ell)) = \pi(\eta)\sum_{\ell=1}^L (u(\eta_{\ell-1},\eta_\ell)+u(\eta_\ell,\eta_{\ell-1})).
	\end{equation*}
	By \eqref{eqn:palrmp-in-cur}, after expanding $f$ by \eqref{eqn:hslrmp-factorised-weight} and rearranging terms, the stationary incoming probability current is
	\begin{align*}
		\pi(\eta)\sum_{\ell=1}^L \bigg(& \dfrac{u(1,\eta_{\ell-1})}{u(\eta_{\ell-1}+1,0)}\dfrac{u(\eta_\ell,0)}{u(1,\eta_\ell-1)}u(\eta_{\ell-1}+1,\eta_\ell-1)\\
		&+\dfrac{u(\eta_{\ell-1},0)}{u(1,\eta_{\ell-1}-1)}\dfrac{u(1,\eta_{\ell})}{u(\eta_{\ell}+1,0)}u(\eta_{\ell}+1,\eta_{\ell-1}-1)\bigg).
	\end{align*}
	Let $G(m,n) = u(m,n) - \dfrac{u(1,n)}{u(n+1,0)}\dfrac{u(m,0)}{u(1,m-1)}u(n+1,m-1)$.
	Then, from the master equation,
	\begin{equation*}
		\sum_{\ell=1}^L (G(\eta_\ell,\eta_{\ell-1})+G(\eta_{\ell-1},\eta_\ell)) = 0.
	\end{equation*}
	So, by \cref{cor:circular-sum-2-real}, there exists $h:\N\to \R$, such that $G(n,m)+G(m,n) = h(n) - h(m)$.
	But, then $h(n) - h(m) = h(m) - h(n)$ for all $n,m \in \N$, that is $h$ is a constant function, and $G(n,m)+G(m,n) = 0$ for all $n,m \in \N$.
	So, for all $n,m \in \N$,
	\begin{equation*}
        \begin{split}
            u(m,n)+u(n,m) = &\dfrac{u(1,n)}{u(n+1,0)}\dfrac{u(m,0)}{u(1,m-1)}u(n+1,m-1)\\
            &+\dfrac{u(1,m)}{u(m+1,0)}\dfrac{u(n,0)}{u(1,n-1)}u(m+1,n-1).
        \end{split}
	\end{equation*}
	This is the same as \eqref{eqn:hslrmp-simplify-u-initial}.
	Thus, by \cref{lem:hslrmp-simplify-u}, we have that $u$ satisfies \eqref{eqn:(h)slrmp-factorised-cond}.
	
    We now prove the converse.
    Consider the case when $u$ satisfies \eqref{eqn:(h)slrmp-factorised-cond}.
    Now, the one-point function in the SLRMP, \eqref{eqn:slrmp-factorised-weight}, when all the rate parameters are $1$, and the one-point function of the homogeneous SLRMP, \eqref{eqn:hslrmp-factorised-weight}, are the same.
    So, the corresponding probability distributions are the same.
    Again, the conditions that $u$ must satisfy in the cases of the SLRMP and the homogeneous SLRMP, \eqref{eqn:(h)slrmp-factorised-cond}, are the same.
    Thus, by \cref{thm:(h)slrmp-factorised} for the SLRMP, the $f$ as given in \eqref{eqn:hslrmp-factorised-weight} is a one-point function corresponding to $u$, and the corresponding distribution is the stationary distribution.
    This completes the proof.
\end{proof}

\begin{cor}
    \label{cor:(h)slrmp-special-u}
    Consider $\vp, \psi: \N_0 \to \R_{> 0}$, where $\vp(0) = 0$, and consider the rate function $u$, such that for all $m,n \in \N_0$,
    \begin{equation}
        \label{eqn:(h)slrmp-special-u}
        u(m,n) = \vp(m)\psi(n).
    \end{equation}
    Then, $u$ satisfies \eqref{eqn:(h)slrmp-factorised-cond}, that is, the stationary distribution corresponding to $u$ in the SLRMP and the homogeneous SLRMP factorises.
    The one-point function corresponding to $u$ in the homogeneous SLRMP is given by
    \begin{equation}
        \label{eqn:(h)slrmp-special-u-weight}
        f(n) = \prod_{i=1}^n \dfrac{\psi(i-1)}{\vp(i)}.
    \end{equation}
\end{cor}
\begin{proof}
    If $u$ is given by \eqref{eqn:(h)slrmp-special-u}, then $u$ satisfies \eqref{eqn:(h)slrmp-factorised-cond}, and by \eqref{eqn:hslrmp-factorised-weight}, a one-point function corresponding to $u$ is
    \begin{equation*}
        \tf(x) = \prod_{i=1}^n \dfrac{\vp(1)\psi(i-1)}{\vp(i)\psi(0)} = \left(\dfrac{\vp(1)}{\psi(0)}\right)^n \prod_{i=1}^n \dfrac{\psi(i-1)}{\vp(i)}.
    \end{equation*}
    Then, by \cref{prop:weight-func-rescale}, $f$ in \eqref{eqn:(h)slrmp-special-u-weight} is a one-point function corresponding to $u$.
\end{proof}

\begin{proof}[Proof of \cref{thm:system-arb-weight}(3)]
    We show the result for the homogeneous SLRMP.
    Similar arguments will prove the result for the SLRMP.
    
    We first consider the case when $g(0) = 1$.
    Define $\phi: \N_0 \to \R_{> 0}$, such that $\phi(n) = g(n+1)/g(n)$, for all $n \in \N_0$, and consider $\vp: \N_0 \to \R_{> 0}$, to be a non-constant function such that $\vp(0) = 0$.
    Consider the rate function $u$, such that for all $m,n \in \N_0$,
    \begin{equation*}
        u(m,n) = \vp(m)\phi(n)\vp(n+1).
    \end{equation*}
    Then, $u(m,0) = \vp(m)\phi(0)\phi(1)$, for $m \geq 1$, is not a constant function, and so, the stationary distribution corresponding to $u$ in the homogeneous PALRMP does not factorise.
    But, by \cref{cor:(h)slrmp-special-u}, the stationary distribution corresponding to $u$ in the homogeneous SLRMP factorises.
    By \eqref{eqn:(h)slrmp-special-u-weight}, a one-point function corresponding to $u$ is
    \begin{equation*}
        f(n) = \prod_{i=1}^n \dfrac{\phi(i-1)\vp(i)}{\vp(i)} = g(n).
    \end{equation*}

    When $g(0) \neq 1$, we proceed as in the proof of \cref{thm:system-arb-weight}, utilising \cref{prop:weight-func-rescale}.
    This completes the proof.
\end{proof}

\section{The discrete Hammersley-Aldous-Diaconis process}
\label{sec:HAD}

The Hammersley-Aldous-Diaconis (HAD) process \cite{hammersley-1972,aldous-diaconis-1995,ferrari-martin-2009} is defined on $\R$ with the configurations being locally finite subsets of $\R$.
The dynamics are governed by a Poisson process of rate $1$ on $\R \times (0,\infty)$.
If the configuration were $\eta$ at time $t^-$ and $(x,t)$ is an element of the Poisson process, then the particle closest to the left of $x$ moves to $x$.
The stationary measure of this process is the Poisson process with positive density.

The HAD process can be similarly described on the circle $S^1$, where the configurations are finite subsets of $S^1$. The stationary distribution in this case, if there were $n$ particles in the system, is the $n$-fold product measure of the uniform distribution on $S^1$.

Identifying $S^1$ with $(0,1]$, for $x_1,x_2,\dots,x_L \in \R_{> 0}$, consider $L$ intervals, where the $\ell$'th interval, denoted $I_\ell$, is,
\[
I_\ell = \left(\dfrac{x_1+\cdots+x_{\ell-1}}{x_1+\cdots+x_L}, 
\dfrac{x_1+\cdots+x_\ell}{x_1+\cdots+x_L}\right].
\]
In the invariant measure of the HAD process on $S^1$ with $N$ particles, each of the particles is in $I_\ell$ with probability ${x_\ell}/{(x_1+\cdots+x_L)}$.
An additional particle comes to the interval $I_\ell$ at time $t$ if there is a point $(r,t)$ in the Poisson process, such that $r$ is between the left endpoint of $I_\ell$ and the leftmost particle in $I_\ell$ at time $t^-$.
If there are $n$ particles in $I_\ell$ at time $t^-$, it can be shown that the expected length of this interval is $x_\ell/((n+1) (x_1+\cdots+x_L))$.
So, an additional particle will arrive with this expected rate.
If the original particle was initially in interval $I_k$, then all the intervals between $I_k$ and $I_\ell$ contain no particles.
As the factor $ 1/({x_1+\cdots+x_L})$ is present for all rates, we can consider the rate at which particles come to $I_\ell$ to be $x_\ell/(n+1)$, if there were $n$ particles in $I_\ell$ at time $t^-$.
The intervals $I_\ell$ correspond to the sites in the TALRMP, and $\phi(n) = 1/(n+1)$ defines the rate function, by \eqref{eqn:palrmp-factorised-cond}.

Henceforth, we consider the case when $\phi(n) = 1/(n+1)$, that is to say, the TALRMP on $\Omega_{L,N}$ with rate
\begin{equation}
    \label{eqn:sdHAD-rate}
    u(m,n) = \begin{cases}
    	\dfrac{1}{n+1},\qquad &\text{if } m \geq 1,\\
    	0, \qquad &\text{if }m=0.
    \end{cases}
\end{equation}
We call this the \emph{discrete HAD process}.
By \cref{prop:weight-func-rescale} and \cref{thm:palrmp-factorised}, $\rho : \Omega_{L,N} \to \R_{> 0}$ given by
\[
    \rho(\eta) = N! \prod_{\ell=1}^L \dfrac{x_\ell^{\eta_\ell}}{\eta_\ell!} =  \binom{N}{\eta_1,\eta_2,\dots,\eta_L}\prod_{\ell=1}^L x_\ell^{\eta_\ell}
\]
is stationary, although not normalized. 
By the multinomial theorem,
\begin{equation}
    \label{eqn:HAD-partition-function}
    \sum_{\eta \in \Omega_{L,N}} \rho(\eta)
    = \sum_{\eta_1+\eta_2+\cdots+\eta_L = N} \binom{N}{\eta_1,\eta_2,\dots,\eta_L}\prod_{\ell=1}^L x_\ell^{\eta_\ell}
    = (x_1+\cdots+x_L)^N.
\end{equation}
So, the stationary probability for $\eta \in \Omega_{L,N}$ is given by
\begin{equation}
    \label{eqn:HAD-stat-prob}
    \pi(\eta) = \dfrac{\rho(\eta)}{(x_1+\cdots+x_L)^N} = \dfrac{N!}{(x_1+\cdots+x_L)^N}\prod_{\ell=1}^L\dfrac{x_\ell}{\eta_\ell!}.
\end{equation}

\begin{thm}
    In the discrete HAD process on $\Omega_{L,N}$, the stationary distribution of particles at a fixed site $\ell$ is the binomial distribution with parameters $N$ and $x_{\ell}/(x_1+\cdots+x_L)$.
\end{thm}

\begin{proof}
    By \cref{prop:trans-inv}, we can consider site $\ell$ to be site $1$.
    Consider $0\leq m \leq L$.
    Then,
    \begin{align*}
        \PP(m&\text{ particles at site 1})
        = \sum_{\substack{\eta\in \Omega_{L,N} \\\eta_1 = m}} \pi(\eta)\\
        &= \sum_{\substack{\eta\in \Omega_{L,N} \\\eta_1 = m}} \left(\dfrac{1}{(x_1+\cdots+x_L)^N}\binom{N}{m,\eta_2,\dots,\eta_L}x_1^m\prod_{\ell=2}^N x_\ell^{\eta_\ell}\right)\\
        &= \dfrac{1}{(x_1+\cdots+x_L)^N}x_1^m\binom{N}{m}\sum_{\substack{\eta\in \Omega_{L,N} \\\eta_1 = m}} \left(\binom{N-m}{\eta_2,\dots,\eta_L}\prod_{\ell=2}^N x_\ell^{\eta_\ell}\right)
    \end{align*}
    Now, each $\eta \in \Omega_{L,N}$ where $\eta_1 = m$ corresponds in a bijective way to an $\teta \in \Omega_{L-1,N-m}$, obtained by disregarding the first site of $\eta$.
    Then, applying \eqref{eqn:HAD-partition-function},
    \begin{equation*}
        \begin{split}
            \sum_{\substack{\eta\in \Omega_{L,N} \\\eta_1 = m}} \left(\binom{N-m}{\eta_2,\dots,\eta_L}\prod_{\ell=2}^N x_\ell^{\eta_\ell}\right)
            &= \sum_{\teta\in \Omega_{L-1,N-m}} \left(\binom{N-m}{\eta_2,\dots,\eta_L}\prod_{\ell=2}^N x_\ell^{\eta_\ell}\right)\\
            &= \sum_{\teta \in \Omega_{L-1,N-m}} \rho(\teta)\\
            &= (x_2+\cdots+x_L)^{N-m}.
        \end{split}
    \end{equation*}
    Thus,
    \begin{equation*}
        \PP(m\text{ particles at site 1}) = \binom{N}{m}\left(\dfrac{x_1}{x_1+\cdots+x_L}\right)^m\left(1-\dfrac{x_1}{x_1+\cdots+x_L}\right)^{N-m},
    \end{equation*}
    completing the proof.
\end{proof}

Consider $\eta \in \Omega_{L,N}$ and sites $k$ and $\ell$, such that there is a transition out of $\eta$ where a particle moves from site $k$ to site $\ell$.
We recall that $\eta^\ell \in \Omega_{L,N+1}$ \eqref{eqn:eta-l} is the configuration which has an additional particle at site $\ell$ of $\eta$.
Then, from \eqref{eqn:HAD-stat-prob}, the stationary probability current due to a transition from $\eta$ to $\eta^{k,\ell}$ is
\begin{equation}
    \label{eqn:HAD-prob-current}
    \pi(\eta) \dfrac{ \, x_\ell}{\eta_\ell+1} = \dfrac{(x_1+\cdots+x_L)}{N+1}\pi(\eta^\ell).
\end{equation}

\begin{prop}
    \label{prop:talrmp-edge-element-correspond}
    Let $L,N \in \N$.
    \begin{enumerate}[leftmargin = *]
        \item[(1)]
        Transitions in the discrete HAD process on $\Omega_{L,N}$ where a particle moves across the edge between site $L$ and site $1$ are in bijection with the configurations of $\Omega_{L,N+1}$ which have at least two sites containing particles.

        \item[(2)]
        Consider a transition in the discrete HAD process where a particle moves across the edge between site $L$ and site $1$,
        and suppose $\xi \in \Omega_{L,N+1}$ is the corresponding configuration from part (1) which has at least two sites containing particles.
        Then the probability current due to this transition is $\pi(\xi) (x_1+\cdots+x_L)/(N+1)$.
    \end{enumerate}
\end{prop}
\begin{proof}
    Consider a transition where the particle moves across the edge between site $L$ and site $1$ in a configuration $\eta$, say.
    Suppose the particle moves from site $k$ to site $\ell$.
    We note that $1\leq \ell < k \leq L$.
    To this transition we associate $\eta^\ell \in \Omega_{L,N+1}$.
    As a particle moves from site $k$ of $\eta$, sites $k$ and $\ell$ of $\eta^\ell$ contain particles.
    So, there are at least two sites of $\eta^\ell$ containing particles.
    In $\eta^\ell$, site $\ell$ is the leftmost site containing particles and site $k$ is the rightmost site containing particles.

    We claim that $\eta^\ell$ is the desired configuration. To show that this is a bijection,
    suppose there is another transition, say out of $\omega \in \Omega_{L,N}$, where a particle moves from site $a$ to site $b$ with $\eta^\ell = \omega^b$.
    If $\ell < b$, then $\omega_\ell = \omega^b_\ell > 0$, in which case, there could not be a transition where a particle moves across the edge between site $1$ and site $L$ and has arrival site as site $b$.
    Similarly, $\ell > b$ cannot occur.
    So, $\ell = b$.
    Similarly, we can show that $k = a$ and so $\eta = \omega$.
    So, this correspondence is injective.

    Consider $\xi \in \Omega_{L,N+1}$, which has at least two sites containing particles, and say that site $\ell$ (resp. site $k$) is the leftmost (resp. rightmost) site which contains a particle.
    Then, there exists $\zeta \in \Omega_{L,N}$, such that $\zeta^\ell = \xi$ and a transition of $\zeta$ where a particle moves to site $\ell$.
    Then, $\xi = \zeta^\ell$ is the configuration corresponding to this transition.
    So, the correspondence is surjective.
    This proves part (1), and part (2) follows from \eqref{eqn:HAD-prob-current}.
\end{proof}

\begin{thm}
	\label{thm:HAD-particle-current}
	On $\Omega_{L,N}$, the current due to particle across an edge is $((x_1+\cdots+x_L)/(N+1))$ times the probability that there are at least two sites containing particles in $\Omega_{L,N+1}$.
\end{thm}

\begin{proof}
    We consider the edge between site $L$ and site $1$.
    For a transition out of $\eta$, where a particle moves across the edge between site $L$ and site $1$, the arrival site $k$ and departure site $\ell$ are such that $1\leq \ell < k \leq L$, and this contributes a current of $\pi(\eta) \, \text{rate}(\eta \to \eta^{k,\ell})$ across the edge between site $L$ and site $1$.
    By \cref{prop:talrmp-edge-element-correspond}(2), each such transition contributes a current of $((x_1+\cdots+x_L)/(N+1))\pi(\eta^\ell)$.
    Also, by \cref{prop:talrmp-edge-element-correspond}(1), each element of $\Omega_{L,N+1}$, where at least two sites contain particles, correspond to such a transition.
    Then, summing over all such transitions, and applying \cref{prop:talrmp-edge-element-correspond}, the current across the edge between site $L$ and site $1$ is $((x_1+\cdots+x_L)/(N+1))$ times the probability that there are at least two sites containing particles in $\Omega_{L,N+1}$
    We note that if a configuration $\xi \in \Omega_{L,N+1}$ is such that all the particles are at a single site, say site $\ell$, then $\rho(\xi) = x_\ell^{N+1}$.
    Thus, the current due to particles across any edge in this system is
    \begin{equation*}
    	\dfrac{(x_1+\cdots+x_L)}{(N+1)}\dfrac{(x_1+\cdots+x_L)^{N+1} - \sum_{\ell=1}^L x_\ell^{N+1} }{(x_1+\cdots+x_L)^{N+1}}.
    \end{equation*}
\end{proof}

\begin{rem}
From \cref{thm:HAD-particle-current}, one can show that the current in the system decreases as the number of particles increases.    
\end{rem}

\begin{rem}
    Consider the TALRMP with $u$ is determined by $\phi$, as in \eqref{eqn:palrmp-factorised-cond}, and $f$ given by \eqref{eqn:palrmp-factorised-weight}.
    The current due to particles in this system would be $(Z^f_{L,N+1}/Z^f_{L,N})$ times the probability that there are at least two sites containing particles in $\Omega_{L,N+1}$, using essentially the same argument
    as in \cref{thm:HAD-particle-current}.
\end{rem}

\section*{Acknowledgements}

We are grateful to James Martin for stimulating discussions.
The authors acknowledge support from the DST FIST program~-~2021 [TPN - 700661]. 

\bibliographystyle{alpha}
\bibliography{lrmp_ref}

\begin{thebibliography}{Ham72}

\bibitem[AD95]{aldous-diaconis-1995}
David Aldous and Persi Diaconis.
\newblock Hammersley's interacting particle process and longest increasing
  subsequences.
\newblock {\em Probability theory and related fields}, 103:199--213, 1995.

\bibitem[CT85]{cocozza-1985}
Christiane Cocozza-Thivent.
\newblock Processus des misanthropes.
\newblock {\em Zeitschrift f{\"u}r Wahrscheinlichkeitstheorie und verwandte
  Gebiete}, 70:509--523, 1985.

\bibitem[EH05]{evans-hanney-2005}
Martin~R Evans and Tom Hanney.
\newblock Nonequilibrium statistical mechanics of the zero-range process and
  related models.
\newblock {\em Journal of Physics A: Mathematical and General}, 38(19):R195,
  2005.

\bibitem[EW14]{evans-waclaw-2014}
Martin~R Evans and Bartek Waclaw.
\newblock Condensation in stochastic mass transport models: beyond the
  zero-range process.
\newblock {\em Journal of Physics A: Mathematical and Theoretical},
  47(9):095001, 2014.

\bibitem[FM09]{ferrari-martin-2009}
Pablo~A Ferrari and James~B Martin.
\newblock Multiclass {H}ammersley-{A}ldous-{D}iaconis process and
  multiclass-customer queues.
\newblock In {\em Annales de l'IHP Probabilit{\'e}s et statistiques},
  volume~45, pages 250--265, 2009.

\bibitem[Ham72]{hammersley-1972}
John~M Hammersley.
\newblock A few seedlings of research.
\newblock In {\em Proc. Sixth Berkeley Symp. Math. Statist. and Probability},
  volume~1, pages 345--394, 1972.

\bibitem[Lig05]{liggett-ips-2005}
Thomas~M. Liggett.
\newblock {\em Interacting particle systems}.
\newblock Classics in Mathematics. Springer-Verlag, Berlin, 2005.
\newblock Reprint of the 1985 original.

\bibitem[Spi70]{spitzer-1970}
Frank Spitzer.
\newblock Interaction of {M}arkov processes.
\newblock {\em Advances in Math.}, 5:246--290 (1970), 1970.

\bibitem[SRB96]{schutz-ramaswamy-barma-1996}
Gunter~M Sch{\"u}tz, Ramakrishna Ramaswamy, and Mustansir Barma.
\newblock Pairwise balance and invariant measures for generalized exclusion
  processes.
\newblock {\em Journal of Physics A: Mathematical and General}, 29(4):837,
  1996.

\end{thebibliography}
	
\end{document}